\documentclass[12pt,reqno]{amsart}

\usepackage{hyperref}
\usepackage{amsthm}
\usepackage{amssymb}
\usepackage{amsxtra}
\usepackage{kopp}  % Gene Kopp's macros
\usepackage{epsfig}
\usepackage{verbatim}
\usepackage{cleveref}

\numberwithin{equation}{section}

\theoremstyle{plain}
\newtheorem{theorem}{Theorem}[section]
\newtheorem{thm}[theorem]{Theorem}
\newtheorem{prop}[theorem]{Proposition}
\newtheorem{cor}[theorem]{Corollary}
\newtheorem{lem}[theorem]{Lemma}

\theoremstyle{definition}
\newtheorem{defn}[theorem]{Definition}

\theoremstyle{remark}

\include{header}

\begin{document}

\title[Indefinite zeta functions]{Indefinite zeta functions}
\author{Gene S. Kopp}

\address{School of Mathematics, University of Bristol, Bristol, UK and Heilbronn Institute for Mathematical Research, Bristol, UK}
\email{gene.kopp@bristol.ac.uk}


\date{\today}

\begin{abstract} 
We define generalised zeta functions associated to indefinite quadratic forms of signature $(g-1,1)$---and more generally, to complex symmetric matrices whose imaginary part has signature $(g-1,1)$---and we investigate their properties. These indefinite zeta functions are defined as Mellin transforms of indefinite theta functions in the sense of Zwegers, which are in turn generalised to the Siegel modular setting. We prove an analytic continuation and functional equation for indefinite zeta functions. We also show that indefinite zeta functions in dimension 2 specialise to differences of ray class zeta functions of real quadratic fields, whose leading Taylor coefficients at $s=0$ are predicted to be logarithms of algebraic units by the Stark conjectures.
\end{abstract}

\maketitle

\section{Introduction}

The Dedekind zeta functions of imaginary quadratic fields are specialisations of real analytic Eisenstein series. For example, for the Gaussian field $K = \Q(i)$ and $\re(s)>1$,
\begin{equation}
\zeta_{\Q(i)}(s) := \sum_{\substack{\aa \leq \Z[i] \\ \aa \neq 0}} \frac{1}{N(\aa)^s} = \frac{1}{4}\sum_{\substack{(m,n)\in \Z^2 \\ (m,n) \neq (0,0)}} \frac{1}{(m^2+n^2)^s} = \frac{\zeta(s)}{2}E(i,s),
\end{equation}
where $E(\tau,s)$ is the real analytic Eisenstein series given for $\im(\tau)>0$ and $\re(s)>1$ by
\begin{equation}
E(\tau,s) := \foh\sum_{\substack{(m,n)\in \Z^2 \\ \gcd(m,n) =1}} \frac{\im(\tau)^s}{\abs{m\tau + n}^{2s}}.
\end{equation}
Placing the discrete set of Dedekind zeta functions into the continuous family of real analytic Eisenstein series allows us to prove many interesting properties of Dedekind zeta functions---for instance, the first Kronecker limit formula is seen to relate $\zeta_K'(0)$ to the value of the Dedekind eta function $\eta(\tau)$ at a CM point.

In this paper, we find a new continuous family of functions, called \textit{indefinite zeta functions} in which ray class zeta functions of \textit{real quadratic} fields sit as a discrete subset. Moreover, we construct \textit{indefinite zeta functions} attached to quadratic forms of signature $(g-1,1)$. In the case $g=2$, norm forms of quadratic number fields give the specialisation of indefinite zeta functions to ray class zeta functions of real quadratic fields.

Indefinite zeta functions have analytic continuation with a functional equation about the line $s=\foh$. This is in contrast to many zeta functions defined by a sum over a cone---such as multiple zeta functions and Shintani zeta functions---which have analytic continuation but no functional equation. Shintani zeta functions are used to give decompositions of ray class zeta functions attached to totally real number fields, which are then used to evaluate those ray class zeta functions (and the closely related Hecke $L$-functions) at non-positive integers \cite{shintanizeta} (see also Neukirch's treatment in \cite{neukirch}, Chapter VII \S 9).
Our specialisation of indefinite zeta functions to ray class zeta functions of real quadratic fields is an alternative to Shintani decomposition that gives a different interpolation between zeta functions attached to real quadratic number fields. 
Indefinite zeta functions also differ from archetypical zeta functions in that they are not (generally) expressed as Dirichlet series.

Indefinite zeta functions are Mellin transforms of indefinite theta functions. Indefinite theta functions were first described by Marie-France Vign\'{e}ras, who constructed modular indefinite theta series with terms weighted by a weight function satisfying a particular differential equation \cite{vig1,vig2}. Sander Zwegers rediscovered indefinite theta functions and used them to construct harmonic weak Maass forms whose holomorphic parts are essentially the mock theta functions of Ramanujan \cite{zwegers}. Zwegers's work triggered an explosion of interest in mock modular forms, with applications to partition identities \cite{bo1}, quantum modular forms and false theta functions \cite{for}, period integrals of the $j$-invariant \cite{duke}, sporadic groups \cite{umbral}, and quantum black holes \cite{blackholes}. Readers looking for additional exposition on these topics may also be interested in the book \cite{bfor} (especially section 8.2) and lecture notes \cite{rolennotes, zagiermock}.

The indefinite theta functions in this paper are a generalisation of those introduced by Zwegers to the Siegel modular setting. Our generalised indefinite zeta functions satisfy a modular transformation law with respect to the Siegel modular group $\Sp_n(\Z)$.

\subsection{Main theorems}

Given a positive definite quadratic form $Q(x_1, \ldots x_g)$ with real coefficients, it is possible to associate a ``definite zeta function'' $\zeta_Q(s)$, sometimes called the Epstein zeta function:
\begin{equation}\label{eq:epstein}
\zeta_Q(s) := \sum_{(n_1,\ldots,n_g) \in \Z^g \setminus \{0\}} \frac{1}{Q(n_1,\ldots,n_g)^s}.
\end{equation}
However, if $Q$ is instead an indefinite quadratic form, the series in \cref{eq:epstein} does not converge. One way to fix this issue it to restrict the sum to a closed subcone $C$ of the double-cone of positivity $\{v \in \R^g : Q(v)>0\}$. This gives rise to a partial indefinite zeta function
\begin{equation}\label{eq:partial}
\zeta^C_Q(s) := \sum_{(n_1,\ldots,n_g) \in C \cap \Z^g} \frac{1}{Q(n_1,\ldots,n_g)^s}.
\end{equation}
However, unlike the Epstein zeta function, this partial zeta function does not satisfy a functional equation.

Our family of \textit{completed indefinite zeta functions} do satisfy a functional equation, although they are not (generally) Dirichlet series.
The completed indefinite zeta function $\widehat{\zeta}^{c_1,c_2}_{p,q}(\Omega,s)$ is defined in terms of the following parameters:
\begin{itemize}
\item a complex symmetric (not necessarily Hermitian) matrix $\Omega = \Omega^\top = iM + N$, with $M = \im(\Omega)$ having signature $(g-1,1)$;
\item ``characteristics'' $p,q \in \R^g$;
\item ``cone parameters'' $c_1,c_2 \in \C^g$ satisfying the inequalities $\ol{c_j}^\top M c_j < 0$;
\item a complex variable $s \in \C$.
\end{itemize}
Due to invariance properties, $\widehat{\zeta}^{c_1,c_2}_{p,q}(\Omega,s)$ remains well-defined with several of the parameters taken to be in quotient spaces:
\begin{itemize}
\item the characteristics on a torus, $p,q \in \R^g/\Z^g$;
\item the cone parameters in complex projective space, $c_1, c_2 \in \Pj^{g-1}(\C)$.
\end{itemize}
The functional equation for the completed indefinite zeta function is given by the following theorem.
\newtheorem*{thm:zetafun}{\normalfont{\textbf{\Cref{thm:zetafun}}}}
\begin{thm}[Analytic continuation and functional equation]\label{thm:zetafun}
The function $\widehat{\zeta}^{c_1,c_2}_{a,b}(\Omega,s)$ may be analytically continued to an entire function on $\C$.
It satisfies the functional equation 
\begin{equation}
\widehat{\zeta}^{c_1,c_2}_{p,q}\left(\Omega,\frac{g}{2}-s\right) =  \frac{e(p^\top q)}{\sqrt{\det(-i\Omega)}} \widehat{\zeta}^{\ol\Omega c_1,\ol \Omega c_2}_{-q,p}\left(-\Omega^{-1},s\right).
\end{equation}
\end{thm}

In the case of real cone parameters, the completed indefinite zeta function has a series expansion that may be viewed as an analogue of the Dirichlet series expansion of a definite zeta function. It decomposes (up to $\Gamma$-factors) as a sum of an \textit{incomplete indefinite zeta function} $\zeta^{c_1,c_2}_{p,q}(\Omega,s)$, which is a Dirichlet series, and correction terms $\xi^{c_j}_{p,q}(\Omega,s)$ that depend only on the cone parameters $c_1$ and $c_2$ separately.
\newtheorem*{thm:splitsy}{\normalfont{\textbf{\Cref{thm:splitsy}}}}
\begin{thm}[Series decomposition]\label{thm:splitsy}
For real cone parameters $c_1,c_2 \in \R^g$, and $\re(s)>1$, the completed indefinite zeta function may be written as
\begin{equation}\label{eq:split}
\widehat{\zeta}^{c_1,c_2}_{p,q}(\Omega,s) = \pi^{-s}\Gamma(s)\zeta^{c_1,c_2}_{p,q}(\Omega,s)-\pi^{-(s+\foh)}\Gamma\left(s+\foh\right)\left(\xi^{c_2}_{p,q}(\Omega,s)-\xi^{c_1}_{p,q}(\Omega,s)\right),
\end{equation}
where $M = \im(\Omega)$,
\begin{align}\label{eq:splitone}
\zeta^{c_1,c_2}_{p,q}(\Omega,s) &= \foh\sum_{n \in \Z^g+q}\left(\sgn(c_1^\top Mn) - \sgn(c_2^\top Mn)\right)e\left(p^\top n\right)Q_{-i\Omega}(n)^{-s},
\end{align}
and
\begin{align}
\xi^{c}_{p,q}(\Omega,s) &= \foh\sum_{\nu \in \Z^g+q} \sgn(c^\top Mn)e\left(p^\top n\right)\left(\frac{(c^\top Mn)^2}{Q_M(c)}\right)^{-s}\nn \\&\ \ \ \ \ \ \ \ \ \ \ \ \ \cdot\hf\left(s,s+\foh,s+1;\frac{2Q_M(c)Q_{-i\Omega}(n)}{(c^\top Mn)^2}\right).
\end{align}
Here, for any complex symmetric matrix $\Lambda$, $Q_\Lambda(v) = v^\top \Lambda v$ denotes the associated quadratic form; also, $\hf$ denotes a hypergeometric function (see \cref{eq:hf}). The summand in \cref{eq:splitone} should always be interpreted as $0$ when $\sgn(c_1^\top Mn) = \sgn(c_2^\top Mn)$; whenever it is nonzero, $\re(Q_{-i\Omega}(n))>0$, and the complex power is interpreted as $Q_{-i\Omega}(n)^{-s} = \exp\left(-s\log\left(Q_{-i\Omega}(n)\right)\right)$ where $\log$ is the principal branch of the logarithm with a branch cut along the negative real axis.
\end{thm}
The series defining the incomplete indefinite zeta function $\zeta^{c_1,c_2}_{p,q}(\Omega,s)$ is a variant of the partial indefinite zeta function \ref{eq:partial}, which may be seen by writing it as
\begin{equation}
\zeta^{c_1,c_2}_{p,q}(\Omega,s) =  \sideset{}{_{}^\ast}\sum_{n \in C^+ \cap (\Z^g + q)} e\left(p^\top n\right)Q_{-i\Omega}(n)^{-s} - \sideset{}{_{}^\ast}\sum_{n \in C^- \cap (\Z^g + q)} e\left(p^\top n\right)Q_{-i\Omega}(n)^{-s},
\end{equation}
where $C^+ = \{v \in \R^g : \sgn(c_2^\top Mv) \leq 0 \leq \sgn(c_1^\top Mv)\}$ and $C^- = \{v \in \R^g : \sgn(c_1^\top Mv) \leq 0 \leq \sgn(c_2^\top Mv)\}$ are subcones of the two components of the double-cone of positivity of $Q_M(v)$, and the notation $\sum^\ast$ means that points on the boundary of the cone are weighted by $\foh$, except for $n=0$, which is excluded.

The indefinite zeta function is defined as a Mellin transform of an \textit{indefinite theta function} (literally, an \textit{indefinite theta null with real characteristics}, see \Cref{defn:indefzeta} and the definitions in \Cref{sec:indeftheta}). Indefinite theta functions were introduced by Sander Zwegers in his PhD thesis \cite{zwegers}. The indefinite theta functions introduced in this paper generalise Zwegers's work to the Siegel modular setting.

Our definition of indefinite zeta functions is in part motivated by an application to the computation of Stark units over real quadratic fields, which will be covered more thoroughly in a companion paper \cite{koppklf}. In special cases, an important symmetry, which we call \textit{$P$-stability}, causes the $\xi^{c_1}$ and $\xi^{c_2}$ terms in \cref{eq:split} to cancel, leaving a Dirichlet series $\zeta^{c_1,c_2}_{p,q}(\Omega,s)$. In the 2-dimensional case ($g=2$), this Dirichlet series is a difference of two ray class zeta functions of an order in a real quadratic field. 
\newtheorem*{thm:special}{\normalfont{\textbf{\Cref{thm:special}}}}
\begin{thm}[Specialization to a ray class zeta function]\label{thm:special}
Let $K$ be a real quadratic number field, and let $\Cl_{\cc\infty_1\infty_2}$ denote the ray class group of $\OO_K$ modulo $\cc\infty_1\infty_2$ (see \cref{eq:Cl}). For each class $A \in \Cl_{\cc\infty_1\infty_2}$ and integral ideal $\bb \in A^{-1}$, there exists a real symmetric matrix $M$ of signature $(1,1)$, along with $c_1, c_2, q \in \C^2$, such that
\begin{equation}
(2\pi N(\bb))^{-s}\Gamma(s)Z_A(s) = \widehat{\zeta}^{c_1,c_2}_{0,q}(iM,s).
\end{equation}
Here, $Z_A(s)$ is the \textit{differenced ray class zeta function} associated to $A$ (see \Cref{defn:diffzeta}).
\end{thm}

This paper is organised as follows. In \Cref{sec:riemann}, we review the theory of Riemann theta functions, which we extend to the indefinite case in \Cref{sec:indeftheta}, generalising Chapter 2 of Zwegers's PhD thesis \cite{zwegers}. In \Cref{sec:defzeta} and \Cref{sec:indefzeta}, we define definite and indefinite zeta functions, respectively, and prove their analytic continuations and functional equations; in particular, we prove \Cref{thm:zetafun}. In \Cref{sec:3series}, we prove a general series expansion for indefinite zeta functions, which is \Cref{thm:splitsy}. In \Cref{sec:quadratic}, we prove that indefinite zeta functions restrict to differences of ray class zeta functions of real quadratic fields, which is \Cref{thm:special}, and we work through an example computation of a Stark unit using indefinite zeta functions.

\subsection{Notation and conventions}

We list for reference the notational conventions used in this paper.
\begin{itemize}
\item $e(z) := \exp(2\pi i z)$ is the complex exponential, and this notation is used for $z \in \C$ not necessarily real.
\item $\HH := \{\tau \in \C : \im(\tau)>0\}$ is the complex upper half-plane.
\item Non-transposed vectors $v \in \C^g$ are always column vectors; the transpose $v^\top$ is a row vector.
\item If $\Lambda$ is a $g \times g$ matrix, then $\Lambda^\top$ is its transpose, and (when $\Lambda$ is invertible) $\Lambda^{-\top}$ is a shorthand for $\left(\Lambda^{-1}\right)^\top$.
\item $Q_\Lambda(v)$ denotes the quadratic form $Q_\Lambda(v) = \foh v^\top \Lambda v$, where $\Lambda$ is a $g \times g$ matrix, and $v$ is a $g \times 1$ column vector.
\item $\left.f(c)\right|_{c=c_1}^{c_2} := f(c_2)-f(c_1)$, where $f$ is any function taking values in an additive group.
\item If $v = \coltwo{v_1}{v_2} \in \C^2$ and $f$ is a function of $\C^2$, we may write $f(v) = f\coltwo{v_1}{v_2}$ rather than $f\left(\coltwo{v_1}{v_2}\right)$.
\item Unless otherwise specified, the logarithm $\log(z)$ is the standard principal branch with $\log(1)=1$ and a branch cut along the negative real axis, and $z^a$ means $\exp(a \log(z))$.
\item Throughout the paper, $\Omega$ will be used to denote a $g \times g$ complex symmetric matrix. We will often need to express $\Omega = iM+N$ where $M,N$ are real symmetric matrices. Matrices denoted by $M$ and $N$ will always have real entries, even when we do not say so explicitly.
\end{itemize}

\subsection{Applications and future work}

A paper in progress will prove a Kronecker limit formula for indefinite zeta functions in dimension 2, which specialises to an analytic formula for Stark units \cite{koppklf}. This formula may be currently be found in the author's PhD thesis \cite{koppthesis}.

While one application of indefinite zeta functions (the new analytic formula for Stark units) is known, we are hopeful that others will be found. We formulate a few research questions to motivate future work.
\begin{itemize}
\item Can one formulate a full modular transformation law for indefinite theta functions $\Theta_\Omega^{c_1,c_2}[f](z,\Omega)$ for some general class of test functions $f$?
\item (How) can indefinite theta functions of arbitrary signature, as introduced by Alexandrov, Banerjee, Manschot, and Pioline \cite{generror}, Nazaroglu \cite{nazaroglu}, and Westerholt-Raum \cite{westerholt}, be generalised to the Siegel modular setting? What do the Mellin transforms of indefinite theta functions of arbitrary signature look like? (Note: Since this paper was posted, a preprint of Roehrig \cite{roehrig} has appeared that answers the first question by providing a description of modular indefinite Siegel theta series by means of a system of differential equations, in the manner of Vign\'{e}ras.)
\item The symmetry property we call $P$-stability is not the only way an indefinite theta function can degenerate to a holomorphic function of $\Omega$; there is also the case when $M=i\Omega$, the quadratic form $Q_M$ factors as a product of two linear forms, and the cone parameters are sent to the boundary of the cone of positivity. How do the associated indefinite zeta functions degenerate in this case? 
\item What can be learned by specialising indefinite zeta functions at integer values of $s$ besides $s=0$ and $s=1$?
\end{itemize}

\section{Definite theta functions}\label{sec:riemann}

In this expository section, we discuss some classical results on (definite) theta functions to provide context for the new results on indefinite theta functions proved in \Cref{sec:indeftheta}. Most of the results in this section may be found in \cite{langabelian}, \cite{milne}, or \cite{thetabook}. The expert may skip most of this section but will need to refer to back \Cref{sec:cansqrt} for conventions for square roots of determinants used in \Cref{sec:morecansqrt}.

Definite theta functions in arbitrary dimension were introduced by Riemann, building on Jacobi's earlier work in one dimension. The work of many mathematicians, including Hecke, Siegel, Schoenberg, and Mumford, further developed the theory of theta functions (including their relationship to zeta functions) and contributed ideas and perspectives used in this exposition.

The definite theta function---or Riemann theta function---of dimension (or genus) $g$ is a function of an elliptic parameter $z$ and a modular parameter $\Omega$.  The elliptic parameter $z$ lives in $\C^g$, but may (almost) be treated as an element of a complex torus $\C^g/\Lambda$, which happens to be an abelian variety.  The parameter $\Omega$ is written as a complex $g \times g$ matrix and lives in the Siegel upper half-space $\HH_g$, whose definition imposes a condition on $M=\im(\Omega)$.

\subsection{Definitions and geometric context}

An \textit{abelian variety} over a field $K$ is a connected projective algebraic group; it follows from this definition that the group law is abelian.  (See \cite{milne} as a reference for all results mentioned in this discussion.)  A \textit{principal polarization} on an abelian variety $A$ is an isomorphism between $A$ and the dual abelian variety $A^\vee$.  Over $K=\C$, every principally polarised abelian variety of dimension $g$ is a complex torus of the form $A(\C) = \C^g/(\Z^g+\Omega\Z^g)$, where $\Omega$ is in the \textit{Siegel upper half-space} (sometimes called the Siegel upper half-plane, although it is a complex manifold of dimension $\frac{g(g+1)}{2})$.
\begin{defn} 
The \textit{Siegel upper half-space} of genus $g$ is defined to be the following open subset of the space $\M_g(\C)$ of symmetric $g \times g$ complex matrices.
\begin{equation}
\HH_g^{(0)} := \HH_g := \{\Omega \in \M_g(\C) : \Omega = \Omega^\top \mbox{ and } \im(\Omega) \mbox{ is positive definite}\}.
\end{equation}
When $g=1$, we recover the usual upper half-plane $\HH_1 = \HH = \{\tau \in \C : \im(\t)>0\}$.
\end{defn}
\begin{defn}
The \textit{definite (Riemann) theta function} is, for $z \in \C^g$ and $\Omega \in \HH_g$,
\begin{equation}
\Theta(z;\Omega) := \sum_{n \in \Z^g} e\left(\foh n^\top \Omega n + n^\top z\right). \label{eq:defdef}
\end{equation}
\end{defn}
\begin{defn}
When $g=1$, the definite theta function is called a \textit{Jacobi theta function} and is denoted by
$
\vartheta(z,\tau) := \Theta([z];[\tau])
$
for $z \in \C$ and $\tau \in \HH$.
\end{defn}
The complex structure on $A(\C)$ determines the algebraic structure on $A$ over $\C$; indeed, the map $A \mapsto A(\C)$ defines an equivalence of categories from the category of abelian varieties over $\C$ to the category of polarizable tori (see Theorem 2.9 in \cite{milne}).
Concretely, theta functions realize the algebraic structure from the analytic. The functions $\Theta(z+t;\Omega)$ for representatives $t \in \C^g$ of $2$-torsion points of $A(\C)$ may be used to define an explicit holomorphic embedding of $A$ as an algebraic locus in complex projective space.  These shifts $t$ are called \textit{characteristics}.  More details may be found in Chapter VI of \cite{langabelian}, in particular pages 104--108.

The positive integer $g$ is sometimes called the ``genus'' because the Jacobian $\Jac(C)$ of an algebraic curve of genus $g$ is a principally polarised abelian variety of dimension $g$.  Not all principally polarized abelian varieties are Jacobians of curves; the question of characterizing the locus of Jacobians of curves inside the moduli space of all principally polarised abelian varieties is known as the \textit{Schottky problem}.

\subsection{The modular parameter and the symplectic group action}

The Siegel upper half-space has a natural action of the real symplectic group. This group, and an important discrete subgroup, are defined as follows.
\begin{defn}
The \textit{real symplectic group} is defined as the set of $2g \times 2g$ real matrices preserving a standard symplectic form.
\begin{equation}
\Sp_{2g}(\R) := \left\{G \in \GL_{2g}(\R) : G^\top \mattwo{0}{-I}{I}{0} G = \mattwo{0}{-I}{I}{0}\right\},
\end{equation}
where $I$ is the $g \times g$ identity matrix. The integer symplectic group is defined by $\Sp_{2g}(\Z) := \Sp_{2g}(\R) \cap \GL_{2g}(\Z)$.
\end{defn}
The real symplectic group acts on the Siegel upper half-space by the \textit{fractional linear transformation action}
\begin{equation}\label{eq:flin}
\mattwo{A}{B}{C}{D} \cdot \Omega := (A\Omega + B)(C\Omega + D)^{-1} \mbox{ for } \mattwo{A}{B}{C}{D} \in \Sp_g(\R).
\end{equation}
We will show in \Cref{prop:S} (specifically, by the case $k=0$) that $\HH_g$ is closed under this action.

\subsection{A canonical square root}\label{sec:cansqrt}

On the Siegel upper half-space $\HH_g$, $\det(-i\Omega)$ has a canonical square root.
\begin{lem}\label{lem:intsq}
Let $\Omega \in \HH_g$.  Then
\begin{equation}\label{eq:intsq}
\left(\int_{x \in \R^g} e\left(\foh x^\top \Omega x\right)\,dx\right)^2 = \frac{1}{\det(-i \Omega)}.
\end{equation}
\end{lem}
\begin{proof}
\Cref{eq:intsq} holds for $\Omega$ diagonal and purely imaginary by reduction to the one-dimensional case $\int_{-\infty}^\infty e^{-\pi a x^2}\,dx = \frac{1}{\sqrt{a}}$.  Consequently, \cref{eq:intsq} holds for any purely imaginary $\Omega$ by a change of basis, using spectral decomposition.

Consider the two sides of \cref{eq:intsq} as holomorphic functions in $\frac{g(g+1)}{2}$ complex variables (the entries of $\Omega$); they agree whenever those $\frac{g(g+1)}{2}$ variables are real.  Because they are holomorphic, it follows by analytic continuation that they agree everywhere.
\end{proof}

\begin{defn}\label{defn:intsq}
\Cref{lem:intsq} provides a canonical square root of $\det(-i\Omega)$:
\begin{align}
\sqrt{\det(-i\Omega)} := \left(\int_{x \in \R^g} e\left(\foh x^\top \Omega x\right)\,dx\right)^{-1}.
\end{align}
Whenever we write ``$\sqrt{\det(-i\Omega)}$'' for $\Omega \in \HH_g$, we will be referring to this square root.
\end{defn}

We will later need to use this square root to evaluate a shifted version of the integral that defines it.
\begin{cor}\label{cor:intsq2}
Let $\Omega \in \HH_g$ and $c \in \C^g$.  Then,
\begin{equation}\label{eq:intsq2}
\int_{x \in \R^g} e\left(\foh (x+c)^\top \Omega (x+c)\right)\,dx = \frac{1}{\sqrt{\det(-i \Omega)}}.
\end{equation}
\end{cor}
\begin{proof}
Fix $\Omega$.  The left-hand side of \cref{eq:intsq2} is constant for $c \in \R^g$, by \Cref{lem:intsq}.  Because the left-hand side is holomorphic in $c$, it is in fact constant for all $c \in \C^g$.
\end{proof}

Note that, if $\Omega \in \HH_g$, then $\Omega$ is invertible and $-\Omega^{-1} \in \HH_g$.  
The latter is true because $-\Omega^{-1} = \mattwo{0}{-I}{I}{0} \cdot \Omega$, where ``$\cdot$'' is the fractional linear transformation action of $\Sp_{2g}(\R)$ on $\HH_g$ defined by \cref{eq:flin}.

The behavior of our canonical square root under the modular transformation $\Omega \mapsto -\Omega^{-1}$ is given by the following proposition.
\begin{prop}
If $\Omega \in \HH_g$, then
$\sqrt{\det(-i\Omega)}\sqrt{\det(i\Omega^{-1})} = 1$.
\end{prop}
\begin{proof}
This follows from \Cref{defn:intsq} by plugging in $\Omega = iI$, because the function given by $\Omega \mapsto \sqrt{\det(-i\Omega)}\sqrt{\det(i\Omega^{-1})}$ is continuous and takes values in $\{\pm 1\}$, and $\HH_g$ is connected.
\end{proof}

\subsection{Transformation laws of definite theta functions}

\begin{prop}\label{prop:zdef}
The definite theta function for $z \in \C^g$ and $\Omega \in \HH_g$ satisfies the following transformation law with respect to the $z$ variable, for $a+\Omega b \in \Z^g + \Omega\Z^g$:
\begin{equation}
\Theta(z+a+\Omega b;\Omega) = e\left(-\foh b^\top \Omega b - b^\top z\right)\Theta(z;\Omega).
\end{equation}
\end{prop}
\begin{proof}
The proof is a straightforward calculation.  It may be found (using slightly different notation) as Theorem 4 on page 8--9 of \cite{thetabook}.
\end{proof}

\begin{thm}\label{thm:omegadef}
The definite theta function for $z \in \C^g$ and $\Omega \in \HH_g$ satisfies the following transformation laws with respect to the $\Omega$ variable, where $A \in \GL_g(\Z)$, $B \in \M_g(\Z)$, $B=B^\top$:
\begin{enumerate}
\item[(1)]
$\Theta(z;A^\top\Omega A) = \Theta(A^{-\top}z;\Omega).$
\item[(2)]
$\Theta(z;\Omega+2B) = \Theta(z;\Omega).$
\item[(3)]
$\Theta(z;-\Omega^{-1}) = \frac{e\left(\foh z^\top \Omega z\right)}{\sqrt{\det(i\Omega^{-1})}}\Theta(\Omega z;\Omega).$
\end{enumerate}
\end{thm}
\begin{proof}
The proof of (1) and (2) is a straightforward calculation.  A more powerful version of this theorem, combining (1)--(3) into a single transformation law, appears as Theorem A on pages 86--87 of \cite{thetabook}.

To prove (3), we apply the Poisson summation formula directly to the theta series.  The Fourier transforms of the terms are given as follows.
\begin{align}
&\int_{\R^g} e\left(Q_\Omega(n)+n^\top z\right)e\left(-n^\top \nu\right)\,dn \nn \\
&= \int_{\R^g} e\left(Q_\Omega(n)+n^\top (z-\nu)\right)\,dn \\
&= e\left(-Q_{-\Omega^{-1}}(z-\nu)\right)\int_{\R^g} e\left(Q_\Omega\left(n+\Omega^{-1}(z-\nu)\right)\right)\,dn \\
&= \frac{e\left(-Q_{-\Omega^{-1}}(z-\nu)\right)}{\sqrt{\det(-i\Omega)}}.
\end{align}
In the last line, we used \Cref{lem:intsq} and \Cref{defn:intsq}.  Now, by the Poisson summation formula,
\begin{align}
\Theta(z;\Omega) &= \sum_{\nu \in \Z^g} \frac{e\left(-Q_{-\Omega^{-1}}(z-\nu)\right)}{\sqrt{\det(-i\Omega)}} \\
&= \frac{e\left(Q_{-\Omega^{-1}}(z)\right)}{\sqrt{\det(-i\Omega)}}\sum_{\nu \in \Z^g} e\left(Q_{-\Omega^{-1}}(\nu)+\nu^\top\Omega^{-1}z\right) \\
&= \frac{e\left(Q_{-\Omega^{-1}}(z)\right)}{\sqrt{\det(-i\Omega)}}\sum_{\nu \in \Z^g} e\left(Q_{-\Omega^{-1}}(\nu)-\nu^\top\Omega^{-1}z\right)  \mbox{ (sending $\nu \mapsto -\nu$)}\\
&= \frac{e\left(-\foh z^\top \Omega^{-1} z\right)}{\sqrt{\det(-i\Omega)}} \Theta\left(-\Omega^{-1}z, -\Omega^{-1}\right).
\end{align}
If $\Omega$ is replaced by $-\Omega^{-1}$, we obtain (3).
\end{proof}

As was mentioned, it is possible to combine all of the modular transformations into a single theorem describing the transformation of $\Omega$ under the action of $\Sp_{2g}(\Z)$,
\begin{equation}
\mattwo{A}{B}{C}{D}\cdot\Omega = (A\Omega+B)(C\Omega+D)^{-1}.
\end{equation}
This rule is already fairly complicated in dimension $g=1$, where the tranformation law involves Dedekind sums.  The general case is done in Chapter III of \cite{thetabook}, with the main theorems stated on pages 86--90.

\subsection{Definite theta functions with characteristics}

There is another notation for theta functions, using ``characteristics,'' and it will be necessary to state the transformation laws using this notation as well.  We replace $z$ with $z = p+\Omega q$ for real variables $p,q \in \R^g$.  The reader is cautioned that the literature on theta functions contains conflicting conventions, and some authors may use notation identical to this one to mean something slightly different.
\begin{defn}
Define the \textit{definite theta null with real characteristics $p,q \in \R^g$}, for $\Omega \in \HH_g$:
\begin{align}
\Theta_{p,q}(\Omega) &:= e\left(\foh q^\top\Omega q + p^\top q\right) \Theta\left(p+\Omega q; \Omega\right).
\end{align}
\end{defn}

The transformation laws for $\Theta_{p,q}(\Omega)$ follow directly from those for $\Theta(z;\Omega)$.

\begin{prop}
Let $\Omega \in \HH_g$ and $p,q \in \R^g$.  The elliptic transformation law for the definite theta null with real characteristics is given by
\begin{equation}
\Theta_{p+a,q+b}(\Omega) = e\left(a^\top(q+b)\right)\Theta_{p,q}(\Omega).
\end{equation}
for $a, b \in \Z^g$.
\end{prop}

\begin{prop}
Let $\Omega \in \HH_g$ and $p,q \in \R^g$.  The modular transformation laws for the definite theta null with real characteristics are given as follows, where $A \in \GL_g(\Z)$, $B \in \M_g(\Z)$, and $B=B^\top$.
\begin{itemize}
\item[(1)] $\Theta_{p,q}(A^\top \Omega A) = \Theta_{A^{-\top}p,Aq}(\Omega)$.
\item[(2)] $\Theta_{p,q}(\Omega+2B) = e(-q^\top Bq)\Theta_{p+2Bq,q}(\Omega)$.
\item[(3)] $\Theta_{p,q}\left(-\Omega^{-1}\right) = \frac{e\left(p^\top q\right)}{\sqrt{\det(i\Omega^{-1})}} \Theta_{-q,p}(\Omega)$.
\end{itemize}
\end{prop}

\section{Indefinite theta functions}\label{sec:indeftheta}

If we allow $\im(\Omega)$ to be indefinite, the series expansion in \cref{eq:defdef} no longer converges anywhere.  We want to remedy this problem by inserting a variable coefficient into each term of the sum.  In Chapter 2 of his PhD thesis \cite{zwegers}, Sander Zwegers found---in the case when $\Omega$ is purely imaginary---a choice of coefficients that preserves the transformation properties of the theta function.  

The results of this section generalise Zwegers's work by replacing Zwegers's indefinite theta function $\vartheta_M^{c_1,c_2}(z,\t)$ by the indefinite theta function $\Theta^{c_1,c_2}[f](z;\Omega)$.  The function has been generalised in the following ways.
\begin{itemize}
\item Replacing $\tau M$ for $\t \in \HH$ and $M \in M_g(\R)$ real symmetric in of signature $(g-1,1)$ by $\Omega \in \HH_g^{(1)}$.  (Adds $\frac{g(g+1)}{2}-1$ real dimensions.)
\item Allowing $c_1,c_2$ to be complex.  (Adds $2g-2$ real dimensions.)
\item Allowing a test function $f(u)$, which must be specialised to $f(u)=1$ for all the modular transformation laws to hold.
\end{itemize}
One motivation for introducing a test function $f$ is to find transformation laws for a more general class of test functions (e.g., polynomials).  We may investigate the behaviour of test functions under modular transformations in future work. However, for the purpose of this paper, only the cases $u \mapsto \abs{u}^r$ will be relevant.

\subsection{The Siegel intermediate half-space}

\begin{defn}
If $M \in \GL_g(\R)$ and $M = M^\top$, the \textit{signature} of $M$ (or of the quadratic form $Q_M$) is a pair $(j,k)$, where $j$ is the number of positive eigenvalues of $M$, and $k$ is the number of negative eigenvalues (so $j+k=g$).
\end{defn}

\begin{defn}\label{defn:intermediate}
For $0 \leq k \leq g$, we define the \textit{Siegel intermediate half-space} of genus $g$ and index $k$ to be
\begin{equation}
\HH_g^{(k)} := \{\Omega \in \M_g(\C) : \Omega = \Omega^\top \mbox{ and } \im(\Omega) \mbox{ has signature } (g-k,k)\}.
\end{equation}
We call a complex torus of the form $T_\Omega := \C^g/(\Z^g+\Omega\Z^g)$ for $\Omega \in \HH_g^{(k)}$, $k \neq 0,g$, an \textit{intermediate torus}.
\end{defn}

Intermediate tori are usually \textit{not} algebraic varieties.  An example of intermediate tori in the literature are the intermediate Jacobians of Griffiths \cite{griffiths1, griffiths2, griffiths3}.  Intermediate Jacobians generalise Jacobians of curves, which are abelian varieties, but those defined by Griffiths are usually not algebraic.  (In contrast, the intermediate Jacobians defined by Weil \cite{weilinter} are algebraic.)

The symplectic group $\Sp_{2g}(\R)$ acts on the set of $g \times g$ complex symmetric matrices by the fractional linear transformation action,
\begin{equation}
\mattwo{A}{B}{C}{D}\cdot \Omega = (A\Omega+B)(C\Omega+D)^{-1}.
\end{equation}

\begin{prop}\label{prop:S}
If $\Omega \in \HH_g^{(k)}$ and $\mattwo{A}{B}{C}{D} \in \Sp_{2g}(\R)$, then $(A\Omega+B)(C\Omega+D)^{-1} \in \HH_g^{(k)}$.  Moreover, the $\HH_g^{(k)}$ are the open orbits of the $\Sp_{2g}(\R)$-action on the set of $g \times g$ complex symmetric matrices.
\end{prop}
\begin{proof}
Trivial for $\mattwo{I}{B}{0}{I}$.  For $\mattwo{A^\top}{0}{0}{A^{-1}}$, this is Sylvester's law of inertia.  For $\mattwo{0}{-I}{I}{0}$, we have $\im(-\Omega^{-1}) = \frac{1}{2i}(-\Omega^{-1}+\ol\Omega^{-1}) = \frac{1}{2i}\ol\Omega^{-1}(-\ol\Omega + \Omega)\Omega^{-1} = \ol\Omega^{-1}\im(\Omega)\Omega^{-1} = \left(\ol\Omega^{-1}\right)^\top\im(\Omega)\Omega^{-1}$, so $\im(-\Omega^{-1})$ and $\im(\Omega)$ have the same signature.  These three types of matrices generate $\Sp_{2g}(\R)$.

Now suppose $\Omega_1, \Omega_2 \in \HH_g^{(k)}$.  There exists a matrix $A \in \GL_g(\R)$ such that $A^\top \im(\Omega_1) A = \im(\Omega_2)$.  For an appropriate choice of real symmetric $B \in \M_g(\R)$, we thus have  $A^\top \Omega_1 A + B = \Omega_2$.  That is,
\begin{equation}
\mattwo{I}{B}{0}{I} \cdot \mattwo{A^\top}{0}{0}{A^{-1}} \cdot \Omega_1 = \Omega_2,
\end{equation}
so $\Omega_1$ and $\Omega_2$ are in the same $\Sp_{2g}(\R)$-orbit.

Thus, the $\HH_g^{(k)}$ are the open orbits of the $\Sp_{2g}(\R)$-action on the set of $g \times g$ symmetric complex matrices.
\end{proof}

\subsection{More canonical square roots}\label{sec:morecansqrt}

From now on, we will focus on the case of index $k=1$, which is signature $(g-1,1)$.  The construction of modular theta series for $k \geq 2$ utilises higher-order error functions arising in string theory \cite{generror}.  More research is needed to develop the higher index theory in the Siegel modular setting. 

\begin{lem}\label{lem:sqrtM}
Let $M$ be a real symmetric matrix of signature $(g-1,1)$.  On the region $R_M=\{z \in \C^g : \ol{z}^\top M z < 0\}$, there is a canonical choice of holomorphic function $g(z)$ such that $g(z)^2=-z^\top M z$.
\end{lem}
\begin{proof}
By Sylvester's law of inertia, there is some $P \in \GL_g^+(\R)$ (i.e., with $\det(P) > 0$) such that $M = P^\top JP$, where
\begin{equation}
J := \left(\begin{array}{ccccc}
-1 & 0 & 0 & \hdots & 0 \\
0 & 1 & 0 & \hdots & 0 \\
0 & 0 & 1& \hdots & 0 \\
\vdots & \vdots & \vdots & \ddots & \vdots \\
0 & 0 & 0 & \hdots & 1
\end{array}\right).
\end{equation}
The region $S := \{(z_2,\ldots,z_g) \in \C^{g-1} : |z_2|^2 + \cdots |z_g|^2 < 1\}$ is simply connected (as it is a solid ball) and does not intersect $\{(z_2,\ldots,z_g) \in \C^{g-1} : z_2^2+\cdots +z_g^2=1\}$ (because, if it did, we'd have $1 = \abs{z_2^2+\cdots +z_g^2} \leq |z_2|^2 + \cdots |z_g|^2 < 1$, a contradiction).  Thus, there exists a unique continuous branch of the function $\sqrt{1-z_2^2-\cdots-z_g^2}$ on $S$ sending $(0,\ldots,0)\mapsto 1$; this function is also holomorphic.
For $z\in R_J$, define
\begin{equation}
g_J(z) := z_1\sqrt{1 - \left(\frac{z_2}{z_1}\right)^2 - \cdots - \left(\frac{z_g}{z_1}\right)^2}.
\end{equation}
This $g_J$ is holomorphic and satisfies $g_J(z)^2 = -z^\top J z$, $g_J(\alpha z) = \alpha g_J(z)$, and $g_J(e_1)=1$ where
\begin{equation}
e_1 := \left(\begin{array}{c}1 \\ 0 \\ \vdots \\ 0\end{array}\right).
\end{equation}
Conversely, if we have a continuous function $g(z)$ satisfying $g(z)^2 = -z^\top J z$ and $g(e_1)=1$, it follows that $g(\alpha z) = \alpha g(z)$, and thus $g(z)=g_J(z)$.

Now, we'd like to define $g_M(z) := g_J(Pz)$, so that we have $g_M(z)^2 = -z^\top M z$.  We need to check that this definition does not depend on the choice of $P$.  Suppose $M = P_1^\top J P_1 = P_2^\top J P_2$ for $P_1,P_2 \in \GL_g^+(\R)$.  So $J = \left(P_2P_1^{-1}\right)^\top J \left(P_2P_1^{-1}\right)$, that is, $P_2P_1^{-1} \in \O(g-1,1)$.  But $\det(P_2P_1^{-1}) = \det(P_2)\det(P_1)^{-1} > 0$, so, in fact, $P_2P_1^{-1} \in \SO(g-1,1)$.  

For any $Q \in \SO(g-1,1)$, we have $g_J(Qe_1)^2 = 1$.  The function $Q \mapsto g_J(Qe_1)$ must be either the constant $1$ or the constant $-1$, because $\SO(g-1,1)$ is connected.  Since $g_J(e_1) = 1$ ($Q=I$), we have $g_J(Qe_1) = 1$ for all $Q \in SO(1,g-1)$.  The function $z \mapsto g_J(Qz)$ is a continuous square root of $-z^\top Jz$ sending $e_1$ to 1, so $g_J(Qz)=g_J(z)$.  Taking $Q = P_2P_1^{-1}$ and replacing $z$ with $P_1 z$, we have $g_J(P_2z)=g_J(P_1z)$, as desired.
\end{proof}

\begin{defn}\label{defn:sqrtM}
If $M$ is a real symmetric matrix of signature $(g-1,1)$, we will write $\sqrt{-z^\top M z}$ for the function $g_M(z)$ in \Cref{lem:sqrtM}.  We may also use similar notation, such as $\sqrt{-\foh z^\top M z} := \frac{1}{\sqrt{2}}\sqrt{-z^\top M z}$.
\end{defn}

\begin{lem}\label{lem:posdefc}
Suppose $M$ is a real symmetric matrix of signature $(g-1,1)$, and $c \in \C^g$ such that $\ol{c}^\top M c < 0$.  Then, $M+M\re\left(\left(-\foh c^\top Mc\right)^{-1}cc^\top\right)M$ is well-defined (that is, $c^\top Mc \neq 0$) and positive definite.
\end{lem}
\begin{proof}
Because $M$ has signature $(g-1,1)$ and $\ol{c}^\top M c < 0$, 
\begin{align}
\left(\ol{c}^\top M c\right)^2 - \abs{c^\top M c}^2 = \det{\mattwo{\ol{c}^\top M c}{c^\top M c}{\ol{c}^\top M \ol{c}}{c^\top M \ol{c}}} < 0.
\end{align}
Thus, $\abs{c^\top M c} > \left(\ol{c}^\top M c\right)^2 > 0$, so $c^\top M c \neq 0$ and $M+M\re\left(\left(-\foh c^\top Mc\right)^{-1}cc^\top\right)M$ is well defined.  Let
\begin{align}
A &:= M+M\re\left(\left(-\foh c^\top Mc\right)^{-1}cc^\top\right)M \\
   &= M - M\left(c^\top Mc\right)^{-1}cc^\top M - M\left(\ol{c}^\top M\ol{c}\right)^{-1}\ol{c}\ol{c}^\top M.
\end{align} 
On the $(g-1)$-dimensional subspace $W = \{w \in \C^g : \ol{c}^\top M w = 0\}$, the sesquilinear form $w \mapsto \ol{w}^\top M w$ is positive definite; this follows from the fact that $\ol{c}^\top Mc<0$, because $M$ has signature $(g-1,1)$.  
For nonzero $w \in W$,
\begin{align}
\ol{w}^\top Aw &= \ol{w}^\top M w - (c^\top Mc)^{-1}(\ol{w}^\top M c)(c^\top M w) - (\ol{c}^\top M\ol{c})^{-1}(\ol{w}^\top M \ol{c})(\ol{c}^\top M w) \\
&= \ol{w}^\top M w - (c^\top Mc)^{-1}(0)(c^\top M w) - (\ol{c}^\top M\ol{c})^{-1}(\ol{w}^\top M \ol{c})(0) \\
&= \ol{w}^\top M w > 0.
\end{align}
Moreover,
\begin{align}
c^\top Aw &= c^\top M w - (c^\top Mc)^{-1}(c^\top M c)(c^\top M w) - (\ol{c}^\top M\ol{c})^{-1}(c^\top M \ol{c})(\ol{c}^\top M w) \\
&= c^\top M w - c^\top M w - (\ol{c}^\top M\ol{c})^{-1}(c^\top M \ol{c})(0) \\
&= 0,
\end{align}
and
\begin{align}
c^\top A\ol{c} &= c^\top M \ol{c} - (c^\top Mc)^{-1}(c^\top M c)(c^\top M \ol{c}) - (\ol{c}^\top M\ol{c})^{-1}(c^\top M \ol{c})(\ol{c}^\top M \ol{c}) \\
&= c^\top M \ol{c} - c^\top M \ol{c} - c^\top M \ol{c} \\
&= - c^\top M \ol{c} \\
&= - \ol{c}^\top M c > 0.
\end{align}
We have now shown that $A$ is positive definite, as it is positive definite on subspaces $W$ and $\C \ol{c}$, and these subspaces span $\C^g$ and are perpendicular with respect to $A$.
\end{proof}

\begin{lem}\label{lem:matrixids}
Let $\Omega = N + iM$ be an invertible complex symmetric $g \times g$ matrix.  Consider $c \in \C^g$ such that $\ol{c}^\top M c < 0$.  The following identities hold:
\begin{itemize}
\item[(1)] $M\Omega^{-1} = \ol{\Omega}\im\left(-\Omega^{-1}\right)$.
\item[(2)] $M-2i M\Omega^{-1} M = \ol\Omega\im\left(-\Omega^{-1}\right)\ol\Omega$.
\item[(3)] $\det\left(-i\left(\Omega - \frac{2i}{c^\top M c}Mcc^\top M\right)\right) = \det(-i\Omega)\frac{c^\top\ol\Omega\im\left(-\Omega^{-1}\right)\ol\Omega c}{c^\top M c}$.
\end{itemize}
\end{lem}
\begin{proof}
Proof of (1):
\begin{align}
M\Omega^{-1} &= \frac{1}{2i}(\Omega - \ol\Omega)\Omega^{-1}
                       = \frac{1}{2i}(I - \ol\Omega\Omega^{-1}) \\
                       &= \ol\Omega\frac{1}{2i}(\ol\Omega^{-1} - \Omega^{-1})
                       = \ol\Omega\im\left(-\Omega^{-1}\right).
\end{align}
Proof of (2):
\begin{align}
M - 2iM\Omega^{-1}M &= M\Omega^{-1}\left(\Omega - 2iM\right) \\
                                 &= \ol\Omega\im\left(-\Omega^{-1}\right)\left(\Omega - (\Omega-\ol\Omega)\right) \mbox{ using (1)} \\
                                 &= \ol\Omega\im\left(-\Omega^{-1}\right)\ol\Omega.
\end{align}
Proof of (3): Note that $\det(I+A) = 1+\Tr(A)$ for any rank 1 matrix $A$.  
Thus,
\begin{align}
&\det\left(-i\left(\Omega - \frac{2i}{c^\top M c}Mcc^\top M\right)\right) \nn \\
&=\det(-i\Omega)\det\left(I + \frac{2i}{c^\top M c}(\Omega Mc)(Mc)^\top\right) \\
&=\det(-i\Omega)\left(1 + \Tr\left(\frac{2i}{c^\top M c}(\Omega Mc)(Mc)^\top\right)\right) \\
&=\det(-i\Omega)\left(1 + \left(\frac{2i}{c^\top M c}c^\top M\Omega^{-1} Mc\right)\right) \\
&=\det(-i\Omega)\frac{-c^\top\left(M - 2i M\Omega^{-1} M\right)c}{-c^\top Mc} \\
&=\det(-i\Omega)\frac{-(\ol\Omega c)^\top\im\left(-\Omega^{-1}\right)(\ol\Omega c)}{-c^\top Mc},
\end{align}
using (2) in the last step.
\end{proof}

\begin{defn}[Canonical square root]
If $\Omega \in \HH_g^{(1)}$, then we define $\sqrt{\det(-i\Omega)}$ as follows.  Write $\Omega = N+iM$ for $N,M\in \M_g(\R)$, and choose any $c$ such that $\ol{c}^\top Mc < 0$.  By \Cref{lem:posdefc}, the matrix $M+M\re\left(\left(-\foh c^\top Mc\right)^{-1}cc^\top\right)M$ is positive definite.  We can also rewrite this matrix as $M+M\re\left(\left(-\foh c^\top Mc\right)^{-1}cc^\top\right)M = \im\left(\Omega - \frac{2i}{c^\top M c}Mcc^\top M\right)$.  By part (3) of \Cref{lem:matrixids},
\begin{equation}
\det\left(-i\left(\Omega - \frac{2i}{c^\top M c}Mcc^\top M\right)\right)
=\det(-i\Omega)\frac{-(\ol\Omega c)^\top\im\left(-\Omega^{-1}\right)(\ol\Omega c)}{-c^\top Mc}.
\end{equation}
We can thus define $\sqrt{\det(-i\Omega)}$ as follows:
\begin{equation}
\sqrt{\det(-i\Omega)} := \frac{\sqrt{-c^\top Mc}\sqrt{\det\left(-i\left(\Omega - \frac{2i}{c^\top M c}Mcc^\top M\right)\right)}}{\sqrt{-(\ol\Omega c)^\top\im\left(-\Omega^{-1}\right)(\ol\Omega c)}},
\end{equation}
where the square roots on the right-hand side are as defined in \Cref{defn:intsq} and \Cref{defn:sqrtM}.  This definition does not depend on the choice of $c$, because $\{c \in \C^g : \ol{c}^\top M c < 0\}$ is connected.
\end{defn}

\subsection{Definition of indefinite theta functions}

\begin{defn}\label{defn:incomplete}
For any complex number $\alpha$ and any entire test function $f$, define the \textit{incomplete Gaussian transform}
\begin{align}
\eE_f(\alpha) &:= \int_{0}^\alpha f(u) e^{-\pi u^2}\,du,
\end{align}
where the integral may be taken along any contour from $0$ to $\alpha$.  
In particular, for the constant functions $\one(u) = 1$, set
\begin{align}
\eE(\alpha) := \eE_{\one}(\alpha) = \int_{0}^\alpha e^{-\pi u^2}\,du = \frac{\alpha}{2|\alpha|}\int_0^{|\alpha|^2} t^{-1/2}e^{-\pi (\alpha/|\alpha|)^2 t}\,dt.
\end{align}
When $\alpha$ is real, define $\eE_f(\alpha)$ for an arbitrary continuous test function $f$:
\begin{align}
\eE_f(\alpha) &:= \int_{0}^\alpha f(u) e^{-\pi u^2}\,du.
\end{align}
\end{defn}

\begin{defn}\label{defn:indeftheta}
Define the \textit{indefinite theta function attached to the test function $f$} to be
\begin{equation}\label{eq:theta}
\Theta^{c_1,c_2}[f](z;\Omega) := \sum_{n \in \Z^g} \left.\eE_f\left(\frac{c^\top\im(\Omega n+z)}{\sqrt{-\foh c^\top \im(\Omega) c}}\right)\right|_{c=c_1}^{c_2} e\left(\frac{1}{2}n^\top\Omega n + n^\top z\right),
\end{equation}
where $\Omega \in \HH_g^{(1)}$, $z \in \C^g$, $c_1, c_2 \in \C^g$, 
$\ol{c_1}^\top M c_1 < 0$, $\ol{c_2}^\top M c_2 < 0$, and $f(\xi)$ is a continuous function of one variable satisfying the growth condition $\log\abs{f(\xi)} = o\left(\abs{\xi}^2\right)$.  If the $c_j$ are not both real, also assume that $f$ is entire.

Also 
define the \textit{indefinite theta function} $\Theta^{c_1,c_2}(z;\Omega) := \Theta^{c_1,c_2}[\one](z;\Omega)$.
\end{defn}

The function $\Theta^{c_1,c_2}(z;\Omega)=\Theta^{c_1,c_2}[\one](z;\Omega)$ is the function we are most interested in, because it will turn out to satisfy a symmetry in $\Omega \mapsto -\Omega^{-1}$. 
We will also show that the functions $\Theta^{c_1,c_2}[u \mapsto |u|^r](z;\Omega)$ are equal (up to a constant) for certain special values of the parameters. 

Before we can prove the transformation laws of our theta functions, we must show that the series defining them converges.

\begin{prop}\label{prop:conv}
The indefinite theta series attached to $f$ (\cref{eq:theta}) converges absolutely and uniformly for $z \in \R^g + iK$, where $K$ is a compact subset of $\R^g$ (and for fixed $\Omega$, $c_1$, $c_2$, and $f$). 
\end{prop}
\begin{proof}
Let $M = \im \Omega$.  
We may multiply $c_1$ and $c_2$ by any complex scalar without changing the terms of the series \cref{eq:theta}, so we assume without loss of generality that $\re(\ol{c_1}^\top M c_2)<0$.

For $\l \in [0,1]$, define the vector $c(\l) = (1-\l) c_1 + \l c_2$ and the real symmetric matrix $A(\l) := M+M\re\left(\left(-\foh c(\l)^\top M c(\l)\right)^{-1}c(\l)c(\l)^\top\right) M$.  
Note that $\ol{c(\l)}^\top Mc(\l) = (1-\l)^2\ol{c_1}^\top Mc_1 + 2\l(1-\l)\re(\ol{c_1}^\top Mc_2) + \l^2\ol{c_2}^\top Mc_2 < 0$ because each term is negative (except when $\l=0$ or $1$, in which case one term is negative and the others are zero).  
By \Cref{lem:posdefc}, $A(\l)$ is well-defined and positive definite for each $\lambda \in [0,1]$.

Consider $(x,\l) \mapsto x^\top A(\l) x$ as a positive real-valued continuous function on the compact set that is the product of the unit ball $\{x^\top x = 1\}$ and the interval $[0,1]$.  It has a global minimum $\e > 0$.

The parametrization $\gamma : [0,1] \to \C$, $\gamma(\l) := \frac{c(\l)^\top(Mn+y)}{\sqrt{-\foh c(\l)^\top M c(\l)}}$, defines a countour from $\frac{c_1^\top(Mn+y)}{\sqrt{-\foh c_1^\top M c_1}}$ to $\frac{c_2^\top(Mn+y)}{\sqrt{-\foh c_2^\top M c_2}}$, so that
\begin{equation}
\left.\eE_f\left(\frac{c^\top(Mn+y)}{-\foh c^\top M c}\right)\right|_{c=c_1}^{c_2} = \int_\gamma f(u)e^{-\pi u^2}\,du.
\end{equation}

We give an upper bound for
\begin{align}
&\max_{\lambda\in [0,1]}\left|e^{-\pi \gamma(\l)^2} e\left(\foh n^\top \Omega n + n^\top z\right)\right| \nn \\
%&= \max_{v\in [0,1]}\left|e^{-\pi \left(c(v)^\top \im(\Omega n+z)\right)^2} e\left(\foh n^\top \Omega n + n^\top z\right)\right| \\
&= e^{\pi y^\top M^{-1} y}\max_{\l\in [0,1]}e^{\frac{-\pi}{-\foh c(\l)^\top M c(\l)} \left(c(\l)^\top M\left(n + M^{-1}y\right)\right)^2} e^{-\pi \left(n + M^{-1}y\right)^\top M \left(n + M^{-1}y\right)} \\
&= e^{\pi y^\top M^{-1} y}\max_{\l\in [0,1]}e^{-\pi \left(n + M^{-1}y\right)^\top A(\l) \left(n + M^{-1}y\right)} \\
&\leq e^{\pi y^\top M^{-1} y}e^{-\pi \e\left\| n + M^{-1}y\right\|^2},
\end{align}
where the vector norm is $\left\|v\right\| := v^\top v$ for $v \in \R^g$.
Thus,
\begin{align}
&\left|\left.\eE_f\left(\frac{c^\top(Mn+y)}{-\foh c^\top M c}\right)\right|_{c=c_1}^{c_2} e\left(\frac{1}{2}n^\top\Omega n + n^\top z\right)\right| \nn \\
&\leq \int_{\gamma} \abs{f(u)}e^{\pi y^\top M^{-1} y}e^{-\pi \e\left\| n + M^{-1}y\right\|^2}\,du \\
%&\leq 2\int_{c_1^\top\im(\Omega n+z)}^{c_2^\top\im(\Omega n+z)} |u|^{2r-1} \left(e^{\pi y^\top M^{-1} y}e^{-\pi \e\left\| n + M^{-1}y\right\|^2}\right)\,du \\
&\leq p(n) e^{-\pi \e\left\| n + M^{-1}y\right\|^2},
\end{align}
where $\log p(n) = o\left(\|n\|^2\right)$.  Thus, the terms of the series are $o\left(e^{-\frac{\pi \e}{2} \left(\|n\|^2 + \|M^{-1}y\|\right)}\right)$, and so the series converges absolutely and uniformly for $x \in \R^g$ and $y \in K$.
\end{proof}

\subsection{Transformation laws of indefinite theta functions}

We will now prove the elliptic and modular transformation laws for indefinite theta functions.  In all of these results, we assume that $z \in \C^g$, $\Omega \in \HH_g^{(1)}$,  $c_j \in \C^g$ satisfying $\ol{c_j}^\top \im(\Omega) c_j$, and $f$ is a function of one variable satisfying the conditions specified in \Cref{defn:indeftheta}.

\begin{prop}\label{prop:z}
The indefinite theta function attached to $f$ satisfies the following transformation law with respect to the $z$ variable, for $a+\Omega b \in \Z^g + \Omega\Z^g$:
\begin{equation}
\Theta^{c_1,c_2}[f](z+a+\Omega b;\Omega) = e\left(-\foh b^\top \Omega b - b^\top z\right)\Theta^{c_1,c_2}[f](z;\Omega).
\end{equation}
\end{prop}
\begin{proof}
By definition,
\begin{align}
&\Theta^{c_1,c_2}[f](z+a+\Omega b;\Omega) \nn \\
&= \sum_{n \in \Z^g} \left.\eE_f\left(\frac{c^\top \im(\Omega n + (z+a+\Omega b))}{-\foh c^\top \im(\Omega)c}\right)\right|_{c=c_1}^{c_2} e\left(Q_{\Omega}(n) + n^\top (z+a+\Omega b)\right).
\end{align}
Because $a \in \Z^g$, $\im(a)$ is zero and $e(n^\top a) = 1$, so
\begin{align}
&\Theta^{c_1,c_2}[f](z+a+\Omega b;\Omega) \nn \\
&= \sum_{n \in \Z^g} \left.\eE_f\left(\frac{c^\top \im(\Omega (n+b) + z)}{-\foh c^\top \im(\Omega)c}\right)\right|_{c=c_1}^{c_2} e\left(Q_{\Omega}(n) + n^\top (z+\Omega b)\right) \\
&= e\left(-\foh b^\top\Omega b\right)\sum_{n \in \Z^g} \left.\eE_f\left(\frac{c^\top \im(\Omega (n+b) + z)}{-\foh c^\top \im(\Omega)c}\right)\right|_{c=c_1}^{c_2} e\left(Q_{\Omega}(n+b) + n^\top z\right) \\
&= e\left(-\foh b^\top\Omega b\right)\sum_{n \in \Z^g} \left.\eE_f\left(\frac{c^\top \im(\Omega n + z)}{-\foh c^\top \im(\Omega)c}\right)\right|_{c=c_1}^{c_2} e\left(Q_{\Omega}(n) + (n-b)^\top z\right) \\
&= e\left(-\foh b^\top \Omega b - b^\top z\right)\Theta[f]^{c_1,c_2}(z;\Omega).
\end{align}
The identity is proved.
\end{proof}

\begin{prop}\label{prop:c}
The indefinite theta function satisfies the following condition with respect to the $c$ variable:
\begin{equation}
\Theta^{c_1,c_3}[f](z;\Omega) = \Theta^{c_1,c_2}[f](z;\Omega) + \Theta^{c_2,c_3}[f](z;\Omega)
\end{equation}
\end{prop}
\begin{proof}
Add the series termwise.
\end{proof}

\begin{thm}\label{thm:omega}
The indefinite theta function satisfies the following transformation laws with respect to the $\Omega$ variable, where $A \in \GL_g(\Z)$, $B \in \M_g(\Z)$, $B=B^\top$:
\begin{enumerate}
\item[(1)]
$\Theta^{c_1,c_2}[f](z;A^\top\Omega A) = \Theta^{Ac_1,Ac_2}[f](A^{-\top}z;\Omega).$
\item[(2)]
$\Theta^{c_1,c_2}[f](z;\Omega+2B) = \Theta^{c_1,c_2}[f](z;\Omega).$
\item[(3)]
In the case where $f(u) = \one(u) = 1$, we have
\begin{equation}
\Theta^{c_1,c_2}(z;-\Omega^{-1}) = \frac{e^{\pi i z^\top \Omega z}}{\sqrt{\det(i\Omega^{-1})}}\Theta^{-\ol{\Omega}^{-1}c_1,-\ol{\Omega}^{-1}c_2}(\Omega z;\Omega).
\end{equation}
\end{enumerate}
\end{thm}
\begin{proof}
The proof of (1) is a direct calculation.
\begin{align}
&\Theta^{c_1,c_2}[f](z;A^\top\Omega A) \nn \\
&=  \sum_{n \in \Z^g} \left.\eE_f\left(\frac{c^\top \im(A^\top\Omega An + z)}{\sqrt{-\foh c^\top \im(\Omega) c}}\right)\right|_{c=c_1}^{c_2} e\left(\frac{1}{2}n^\top A^\top\Omega A n + n^\top z\right) \\
&=  \sum_{m \in \Z^g} \left.\eE_f\left(\frac{c^\top \im(A^\top\Omega m + z)}{\sqrt{-\foh c^\top \im(\Omega) c}}\right)\right|_{c=c_1}^{c_2} e\left(\frac{1}{2}m^\top \Omega m + \left(A^{-1}m\right)^\top z\right)
\end{align}
by the change of basis $m=An$, so
\begin{align}
&\Theta^{c_1,c_2}[f](z;A^\top\Omega A) \nn \\
&=  \sum_{m \in \Z^g} \left.\eE_f\left(\frac{(Ac)^\top \im(\Omega m + A^{-\top} z)}{\sqrt{-\foh c^\top \im(\Omega) c}}\right)\right|_{c=c_1}^{c_2} e\left(\frac{1}{2}m^\top \Omega m + m^\top A^{-\top}z\right) \\
&= \Theta^{Ac_1,Ac_2}[f](A^{-\top}z;\Omega).
\end{align}

The proof of (2) is also a direct calculation.
\begin{align}
&\Theta^{c_1,c_2}[f](z;\Omega + 2B) \nn \\
&=  \sum_{n \in \Z^g} \left.\eE_f\left(\frac{c^\top \im((\Omega+2B)n + z)}{\sqrt{-\foh c^\top \im(\Omega) c}}\right)\right|_{c=c_1}^{c_2} e\left(\frac{1}{2}n^\top (\Omega+2B) n + n^\top z\right) \\
&=  \sum_{n \in \Z^g} \left.\eE_f\left(\frac{c^\top (\im((\Omega)n + z))+2\im(B)n}{\sqrt{-\foh c^\top \im(\Omega) c}}\right)\right|_{c=c_1}^{c_2} e\left(Q_{\Omega}(n) + n^\top B n + n^\top z\right) \\
&=  \sum_{n \in \Z^g} \left.\eE_f\left(\frac{c^\top \im((\Omega)n + z)}{\sqrt{-\foh c^\top \im(\Omega) c}}\right)\right|_{c=c_1}^{c_2} e\left(Q_{\Omega}(n) + n^\top z\right) \\
&= \Theta^{c_1,c_2}[f](z;\Omega);
\end{align}
where $e\left(n^\top B n\right)=1$ because the $n^\top B n$ are integers, and $\im(B)=0$ because $B$ is a real matrix.

The proof of (3) is more complicated, and, like the proof of the analogous property for definite (Jacobi and Riemann) theta functions, uses Poisson summation.  The argument that follows is a modification of the argument that appears in the proof of Lemma 2.8 of Zwegers's thesis \cite{zwegers}.

We will find a formula for the Fourier transform of the terms of our theta series.  Most of the work is done in the calculation of the integral that follows.  In this calculation, $M = \im\Omega$, and $z = x+iy$ for $x,y\in \C^g$.  The differential operator $\Del_x$ is a row vector with entries $\frac{\partial}{\partial x_j}$, and similarly for $\Del_n$.

\begin{align}
&\Del_x\left(\int_{n \in \R^g} \left.\eE\left(\frac{c^\top M n + c^\top y}{\sqrt{-\foh c^\top M c}}\right)\right|_{c=c_1}^{c_2} e\left(Q_\Omega\left(n+\Omega^{-1}z\right)\right)\,dn\right) \nn \\
&= \int_{n \in \R^g} \left.\eE\left(\frac{c^\top M n + c^\top y}{\sqrt{-\foh c^\top M c}}\right)\right|_{c=c_1}^{c_2} \Del_x \left(e\left(Q_\Omega\left(n+\Omega^{-1}z\right)\right)\right)\,dn \\
&= \left(\int_{n \in \R^g} \left.\eE\left(\frac{c^\top M n + c^\top y}{\sqrt{-\foh c^\top M c}}\right)\right|_{c=c_1}^{c_2} \Del_n \left(e\left(Q_\Omega\left(n+\Omega^{-1}z\right)\right)\right)\,dn \right)\Omega^{-1}\\
&= \left(-\int_{n \in \R^g} \left.\Del_n \left(\eE\left(\frac{c^\top M n + c^\top y}{\sqrt{-\foh c^\top M c}}\right)\right)\right|_{c=c_1}^{c_2} e\left(Q_\Omega\left(n+\Omega^{-1}z\right)\right)\,dn\right)\Omega^{-1} \label{parts}\\
%&= \left.\left(k\int_{n \in \R^g} e\left(\frac{i}{-c^\top M c}\left(c^\top \im(\Omega) n + c^\top y\right)^2\right) e\left(Q_\Omega\left(n+\Omega^{-1}z\right)\right)\,dn\right)c^\top M\Omega^{-1}\right|_{c=c_1}^{c_2} \\
&= \left.\left(k\int_{n \in \R^g} e\left(\frac{i}{-c^\top M c}\left(c^\top \im(\Omega) n\right)^2\right) e\left(Q_{\Omega}(n+a_z)\right)\,dn\right)c^\top M\Omega^{-1}\right|_{c=c_1}^{c_2},
\end{align}
where $k := \frac{-2}{\sqrt{-\foh c^\top M c}} \in \C$, $a_z := \Omega^{-1}z-M^{-1}y \in \C^g$, and integration by parts was used in \cref{parts}.  Continuing the calculation,
\begin{align}
&\Del_x\left(\int_{n \in \R^g} \left.\eE\left(\frac{c^\top M n + c^\top y}{\sqrt{-\foh c^\top M c}}\right)\right|_{c=c_1}^{c_2} e\left(Q_\Omega\left(n+\Omega^{-1}z\right)\right)\,dn\right) \nn \\
&= \left.k\left(\int_{n \in \R^g} e\left(Q_{\Omega-\frac{2i}{c^\top M c}Mcc^\top M}(n)+a^\top \Omega n+\foh a^\top \Omega a\right)\,dn\right)c^\top M \Omega^{-1}\right|_{c=c_1}^{c_2} \\
&= \left.k e\left(-\foh a^\top \Omega\left(\Omega-\frac{2i}{c^\top M c}Mcc^\top M\right)^{-1}\Omega a+\foh a^\top \Omega a\right)I^{(c)} c^\top M \Omega^{-1}\right|_{c=c_1}^{c_2},
\end{align}
where
\begin{align}
I^{(c)} &:= \int_{n \in \R^g} e\left(Q_{\Omega-\frac{2i}{c^\top M c}Mcc^\top M}\left(n+\left(\Omega-\frac{2i}{c^\top M c}Mcc^\top M\right)^{-1}\Omega a\right)\right)\,dn \\
&= \frac{1}{\det{\sqrt{-i\left(\Omega-\frac{2i}{c^\top M c}Mcc^\top M\right)}}}
\end{align}
by \Cref{lem:intsq}.

We can check (by multiplication) that 
\begin{equation}
\left(\Omega - \frac{2i}{c^\top M c}Mcc^\top M\right)^{-1} = \Omega^{-1} - \frac{2i}{c^\top Mc - 2ic^\top M\Omega^{-1}Mc}\Omega^{-1}Mcc^\top M\Omega^{-1}.
\end{equation}
Thus,
\begin{equation}
\Omega-\Omega\left(\Omega-\frac{2i}{c^\top M c}Mcc^\top M\right)^{-1}\Omega
= \frac{2i}{c^\top Mc - 2ic^\top M\Omega^{-1}Mc}Mcc^\top M.
\end{equation}

Now compute, using \Cref{lem:matrixids},
\begin{align}
Ma &= M\Omega^{-1}z-y 
= \ol\Omega\im\left(-\Omega^{-1}\right)z-y
= \ol\Omega\left(\im\left(-\Omega^{-1}\right)z-\ol\Omega^{-1}y\right)\\
&= \frac{1}{2i}\ol\Omega\left(\left(-\Omega^{-1}+\ol\Omega^{-1}\right)z-\ol\Omega^{-1}(z-\ol{z})\right)
= \frac{1}{2i}\ol\Omega\left(-\Omega^{-1}z+\ol\Omega^{-1}\ol{z}\right)\\
&= \ol\Omega\im\left(-\Omega^{-1}z\right).
\end{align}
Also by \Cref{lem:matrixids}, $M - 2i M\Omega^{-1}M = \ol\Omega\im\left(-\Omega^{-1}\right)\ol\Omega$, and
\begin{align}
\sqrt{\det\left(-i\left(\Omega-\frac{2i}{c^\top M c}Mcc^\top M\right)\right)}
&= \sqrt{\det(-i\Omega)}\frac{\sqrt{-c^\top\ol\Omega\im\left(-\Omega^{-1}\right)\ol\Omega c}}{\sqrt{-c^\top M c}}.
\end{align}
We have now shown that
\begin{align}
%&\Del_x \left(e\left(\foh (z+\nu)^\top\Omega^{-1}(z+\nu)\right)\int_{n \in \R^g} f(n)e(-n^\top \nu)\,dn\right) \\
&\Del_x\left(\int_{n \in \R^g} \left.\eE\left(\frac{c^\top\im\left(\Omega n+z\right)}{\sqrt{-\foh c^\top M c}}\right)\right|_{c=c_1}^{c_2} e\left(Q_\Omega\left(n+\Omega^{-1}z\right)\right)\,dn\right) \nn \\
%&= k e\left(\frac{i}{(\ol\Omega c)\im(-\Omega^{-1})(\ol\Omega c)}(c^\top Ma)^2\right)\frac{\sqrt{-\foh c^\top M c}}{\sqrt{\det(i\Omega)}\sqrt{-\foh(\ol\Omega c)\im(-\Omega^{-1})(\ol\Omega c)}}c^\top M \Omega^{-1} \\
&= \left.\frac{-2e\left(\frac{i}{(\ol\Omega c)\im(-\Omega^{-1})(\ol\Omega c)}(c^\top Ma)^2\right)}{\sqrt{\det(-i\Omega)}\sqrt{-\foh(\ol\Omega c)\im(-\Omega^{-1})(\ol\Omega c)}}
c^\top M \Omega^{-1}\right|_{c=c_1}^{c_2} \\
&= \left.\frac{-2e\left(\frac{i}{(\ol\Omega c)\im(-\Omega^{-1})(\ol\Omega c)}(c^\top Ma)^2\right)}{\sqrt{\det(-i\Omega)}\sqrt{-\foh(\ol\Omega c)\im(-\Omega^{-1})(\ol\Omega c)}}
(\Omega c)^\top \im(\Omega^{-1})\right|_{c=c_1}^{c_2} \\
&= \left.\frac{1}{\sqrt{\det(-i\Omega)}}\Del_x \eE\left(\frac{(\ol\Omega c)^\top\left(\im(-\Omega^{-1})n+\im(-\Omega^{-1}z)\right)}{\sqrt{-\foh(\ol\Omega c)\im(-\Omega^{-1})(\ol\Omega c)}}\right)\right|_{c=c_1}^{c_2}.
\end{align}
Define the following function on $\C^g$,
\begin{align}
C(z) := &\int_{n \in \R^g} \left.\eE\left(\frac{c^\top\im\left(\Omega n+z\right)}{\sqrt{-\foh c^\top \Omega c}}\right)\right|_{c=c_1}^{c_2} e\left(Q_\Omega\left(n+\Omega^{-1} z\right)\right)\,dn \nn \\
&- \frac{1}{\sqrt{\det(i\Omega)}}\left.\eE\left(\frac{(\ol\Omega c)^\top\im(-\Omega^{-1}z)}{\sqrt{-\foh(\ol\Omega c)\im(-\Omega^{-1})(\ol\Omega c)}}\right)\right|_{c=c_1}^{c_2},
\end{align}
suppressing the dependence of $C(z)$ on $\Omega$ and $c$.  We have just shown that $\Del_x C(z) = 0$, so $C(z+a) = C(z)$ for any $a \in \R^g$.  By inspection, $C(z+\Omega^{-1} b) = C(z)$ for any $b \in \R^g$.  It follow from both of these properties that $C(z)$ is constant.  Moreover, by inspection, $C(-z)=-C(z)$; therefore, $C(z)=0$.  In other words,
\begin{align}
&\int_{n \in \R^g} \left.\eE\left(\frac{c^\top\im\left(\Omega n+z\right)}{\sqrt{-\foh c^\top \Omega c}}\right)\right|_{c=c_1}^{c_2} e\left(Q_\Omega\left(n+\Omega^{-1} z\right)\right)\,dn \nn \\
&= \frac{1}{\sqrt{\det(-i\Omega)}}\left.\eE\left(\frac{(\ol\Omega c)^\top\im(-\Omega^{-1}z)}{\sqrt{-\foh(\ol\Omega c)\im(-\Omega^{-1})(\ol\Omega c)}}\right)\right|_{c=c_1}^{c_2}.\label{eq:moose}
\end{align}

Now set $g(z) := \Theta^{c_1,c_2}(z;\Omega)$, which has Fourier coefficients
\begin{equation}
c_n(g)(z) = \left.\eE\left(\frac{c^\top\im\left(\Omega n+z\right)}{\sqrt{-\foh c^\top \Omega c}}\right)\right|_{c=c_1}^{c_2}e\left(\foh n^\top \Omega n + n^\top z\right).
\end{equation}
By plugging in $z-\nu$ for $z$ in \cref{eq:moose} and multiplying both sides by $e\left(-\foh (z-\nu)^\top\Omega^{-1}(z-\nu)\right)$, we obtain the following expression for the Fourier coefficients of $\widehat{g}$:
\begin{align}
c_{\nu}\left(\widehat{g}\right)(z)
&=\int_{n \in \R^g} \left.\eE\left(\frac{c^\top\im\left(\Omega n+z\right)}{\sqrt{-\foh c^\top \Omega c}}\right)\right|_{c=c_1}^{c_2}e\left(\foh n^\top \Omega n + n^\top z\right)e(-n^\top\nu)\,dn \\
&= \frac{e\left(-\foh (z-\nu)^\top\Omega^{-1}(z-\nu)\right)}{\sqrt{\det(-i\Omega)}}\left.\eE\left(\frac{(\ol\Omega c)^\top\im(-\Omega^{-1}\nu-\Omega^{-1}z)}{\sqrt{-\foh(\ol\Omega c)\im(-\Omega^{-1})(\ol\Omega c)}}\right)\right|_{c=c_1}^{c_2} \\
&= \frac{e\left(-\foh z^\top\Omega^{-1}z\right)}{\sqrt{\det(-i\Omega)}}\left.\eE\left(\frac{(\ol\Omega c)^\top\im(-\Omega^{-1}(-\nu)-\Omega^{-1}z)}{\sqrt{-\foh(\ol\Omega c)\im(-\Omega^{-1})(\ol\Omega c)}}\right)\right|_{c=c_1}^{c_2}\\&\hspace{80pt}\cdot e\left(\foh \nu^\top(-\Omega^{-1})\nu+(-\nu)^\top(-\Omega^{-1}z)\right).
\end{align}
It follows by Poisson summation that
\begin{align}
\Theta^{c_1,c_2}(z;\Omega) 
&= \sum_{\nu \in \Z^g} c_{\nu}\left(\widehat{g}\right)(z) \\
&= \frac{e\left(-\foh z^\top\Omega^{-1}z\right)}{\sqrt{\det(-i\Omega)}}\Theta^{\ol{\Omega}c_1,\ol{\Omega}c_2}\left(-\Omega^{-1} z;-\Omega^{-1}\right).
\end{align}
We obtain (3) by replacing $\Omega$ with $-\Omega^{-1}$.
\end{proof}

\subsection{Indefinite theta functions with characteristics}

Now we restate the transformation laws using ``characteristics'' notation, which will be used when we define indefinite zeta functions in \Cref{sec:indefzeta}.
\begin{defn}\label{defn:characteristics}
Define the \textit{indefinite theta null with characteristics $p,q \in \R^g$}:
\begin{align}
\Theta^{c_1,c_2}_{p,q}[f](\Omega) &:= e^{2\pi i\left(\foh q^\top\Omega q + p^\top q\right)} \Theta^{c_1,c_2}[f]\left(p+\Omega q; \Omega\right); \\
%\Theta^{c_1,c_2}_{p,q,r}(\Omega) &:= e^{2\pi i\left(\foh q^\top\Omega q + p^\top q\right)} \Theta^{c_1,c_2}_r\left(p+\Omega q; \Omega\right); \\
\Theta^{c_1,c_2}_{p,q}(\Omega) &:= e^{2\pi i\left(\foh q^\top\Omega q + p^\top q\right)} \Theta^{c_1,c_2}\left(p+\Omega q; \Omega\right).
\end{align}
\end{defn}

\noindent The transformation laws for $\Theta^{c_1,c_2}_{p,q}[f](\Omega)$ follow from the transformation laws for $\Theta^{c_1,c_2}[f](z;\Omega)$.  

\begin{prop}
The elliptic transformation law for the indefinite theta null with characteristics is:
\begin{equation}
\Theta^{c_1,c_2}_{p+a,q+b}[f](\Omega) = e(a^\top(q+b))\Theta^{c_1,c_2}_{p,q}[f](\Omega).
\end{equation}
\end{prop}

\begin{prop}
The modular transformation laws for the indefinite theta null with characteristics are as follows.
\begin{itemize}
\item[(1)] $\Theta^{c_1,c_2}_{p,q}[f](A^\top \Omega A) = \Theta^{Ac_1,Ac_2}_{A^{-\top}p,Aq}[f](\Omega)$.
\item[(2)] $\Theta^{c_1,c_2}_{p,q}[f](\Omega+2B) = e(-q^\top Bq)\Theta^{c_1,c_2}_{p+2Bq,q}[f](\Omega)$.
\item[(3)] $\Theta^{c_1,c_2}_{p,q}(-\Omega^{-1}) = \frac{e(p^\top q)}{\sqrt{\det(i\Omega^{-1})}} \Theta^{-\ol\Omega^{-1} c_1,-\ol\Omega^{-1}c_2}_{-q,p}(\Omega)$.
\end{itemize}
\end{prop}

\subsection{$P$-stable indefinite theta functions}

We now introduce a special symmetry that may be enjoyed by the parameters $(c_1,c_2,z,\Omega)$, which we call \textit{$P$-stability}.  In this section, $c_1,c_2$ will always be real vectors.

\begin{defn}\label{defn:stability}
Let $P \in \GL_g(\Z)$ be fixed.  Let $z \in \C^g$, $\Omega \in \HH_g^{(1)}$, $c_1, c_2 \in \R^g$ satisfying $c_j^\top \im(\Omega) c_j < 0$.  The quadruple $(c_1,c_2,z,\Omega)$ is called \textit{$P$-stable} if $P^\top \Omega P = \Omega$, $Pc_1 = c_2$, and $P^\top z \con z \Mod{\Z^2}$.
\end{defn}

Remarkably, $P$-stable indefinite theta functions attached to $f(u) = |u|^r$ turn out to be independent of $r$ (up to a constant factor).  
\begin{thm}[$P$-Stability Theorem]\label{thm:stability}
Set $\Theta^{c_1,c_2}_r(z;\Omega) := \frac{\pi^{\frac{r+1}{2}}}{\Gamma\left(\frac{r+1}{2}\right)}\Theta^{c_1,c_2}[f](z;\Omega)$ when $f(u) = \abs{u}^r$ for $\re(r)>-1$.
If $(c_1,c_2,z,\Omega)$ is $P$-stable, then $\Theta_r^{c_1,c_2}(z;\Omega)$ is independent of $r$.
\end{thm}
\begin{proof}
Let $M=\im(\Omega)$ and $y=\im(z)$.
If $\alpha \in \R$ and $\re(r)>1$, then
\begin{align}
\eE_r(\alpha) 
&= \int_{0}^\alpha |u|^r e^{-\pi u^2}\,du \\
&= \sgn(\alpha)\int_{0}^{\abs{\alpha}} u^r e^{-\pi u^2}\,du \\
&= -\frac{\sgn(\alpha)}{2\pi}\int_{0}^{\abs{\alpha}} u^{r-1} \,d\left(e^{-\pi u^2}\right) \\
&= -\frac{\sgn(\alpha)}{2\pi}\left(\left.u^{r-1} e^{-\pi u^2}\right|_{u=0}^{\abs{\alpha}}-\int_{0}^{\abs{\alpha}}e^{-\pi u^2} \,d\left(u^{r-1}\right)\right) \\
&= -\frac{\sgn(\alpha)}{2\pi}\left(\abs{\alpha}^{r-1} e^{-\pi \alpha^2}-(r-1)\int_{0}^{\abs{\alpha}}u^{r-2}e^{-\pi u^2} \,du\right) \\
&= \frac{1}{2\pi}\left(-\sgn(\alpha)\abs{\alpha}^{r-1} e^{-\pi \alpha^2}+(r-1)\eE_{r-2}(\alpha)\right).
\end{align}
Let $\a_n^c = \frac{c^\top\im(\Omega n+z)}{\sqrt{-Q_M(c)}}$.  
Set $A^c := M+M\re\left(\left(-Q_M(c)\right)^{-1}cc^\top\right) M$, so that $A^{c_1}$ and $A^{c_2}$ are positive definite, as in the proof of \Cref{prop:conv}. 
Thus,
\begin{equation}
\Theta_r^{c_1,c_2}(z;\Omega) = -\frac{\pi^{r/2}}{\Gamma\left(\frac{r+1}{2}\right)}S + \Theta_{r-2}^{c_1,c_2}(z;\Omega),
\end{equation}
where
\begin{align}
S &= \sum_{n \in \Z^g} \left.\sgn\left(\a_n^c\right)\abs{\a_n^c}^{r-1} \exp\left(-\pi\left(\a_n^c\right)^2\right)\right|_{c=c_1}^{c_2} e\left(\frac{1}{2}n^\top\Omega n + n^\top z\right).% \\
%&= \sum_{n \in \Z^g} \left(\left.\sgn\left(\a_n^c\right)\abs{\a_n^c}^{r-1}\exp\left(-\pi (n+M^{-1}y)^\top A^{c}(n+M^{-1}y)\right)\right|_{c=c_1}^{c_2}\right).
\end{align}
The $c_1$ and $c_2$ terms in this sum decay exponentially, because
\begin{equation}
\abs{\exp\left(-\pi\left(\a_n^c\right)^2\right) e\left(\frac{1}{2}n^\top\Omega n + n^\top z\right)}
= \exp\left(-2\pi Q_{A^{c}}\left(n+M^{-1}y)\right)\right).
\end{equation}
Thus, the series may be split as a sum of two series:
\begin{align}
S
&= \sum_{n \in \Z^g} \sgn\left(\a_n^{c_2}\right)\abs{\a_n^{c_2}}^{r-1}\exp\left(-\pi\left(\a_n^{c_2}\right)^2\right) e\left(\frac{1}{2}n^\top\Omega n + n^\top z\right) \nn \\
&\hspace{10pt}-
\sum_{n \in \Z^g} \sgn\left(\a_n^{c_1}\right)\abs{\a_n^{c_1}}^{r-1}\exp\left(-\pi\left(\a_n^{c_1}\right)^2\right) e\left(\frac{1}{2}n^\top\Omega n + n^\top z\right).
\end{align}
Now we use the $P$-symmetry to show that these two series are, in fact, equal. 
Note that $\im(P^\top z) = \im(z)$ because $P^\top z \con z \Mod{\Z^2}$, so
\begin{align}
\a_{Pn}(c_2) 
&= \frac{(Pc_1)^\top\im(\Omega Pn+z)}{\sqrt{-Q_M(Pc_1)}} \\
&= \frac{c_1^\top\im(P^\top\Omega Pn+P^\top z)}{\sqrt{-Q_{P^\top MP}(c_1)}} \\
&= \frac{c_1^\top\im(\Omega n+z)}{\sqrt{-Q_{M}(c_1)}} \\
&= \alpha_n(c_1).
\end{align}
Moreover, 
\begin{align}
\frac{1}{2}(Pn)^\top \Omega (Pn) + (Pn)^\top z &= \frac{1}{2}n^\top (P^\top \Omega P)n + n^\top (P^\top z) \\ &\con \frac{1}{2}n^\top \Omega n + n^\top z \Mod{\Z^2}.
\end{align}
Thus, we may substitute $Pn$ for $n$ in the first series (involving $c_2$) to obtain the second (involving $c_1$).

We've now shown the periodicity relation
\begin{equation}
\Theta_r^{c_1,c_2}(z;\Omega) = \Theta_{r-2}^{c_1,c_2}(z;\Omega).
\end{equation}
Note that this identity provides an analytic continuation of $\Theta_r^{c_1,c_2}(z,\Omega)$ to the entire $r$-plane.  To show that it is constant in $r$, we will show that it is bounded on vertical strips in the $r$-plane.  As in the proof of \Cref{prop:conv}, bound $(x,\l) \mapsto x^\top A(\l) x$, considered as a positive real-valued continuous function on the product of the unit ball $\{x^\top x = 1\}$ and the interval $[0,1]$, from below by its global minimum $\e > 0$.  Thus,
\begin{align}
&\left|\left.\eE_r\left(\frac{c^\top(Mn+y)}{\sqrt{-\foh c^\top M c}}\right)\right|_{c=c_1}^{c_2} e\left(\frac{1}{2}n^\top\Omega n + n^\top z\right)\right| \nn \\
&\leq \abs{\int_{\frac{c_1^\top\im(\Omega n+z)}{\sqrt{-\foh c_1^\top \im(\Omega) c_1}}}^{\frac{c_2^\top\im(\Omega n+z)}{\sqrt{-\foh c_2^\top \im(\Omega) c_2}}} \abs{u}^{\re(r)}\,du} e^{\pi y^\top M^{-1} y}e^{-\pi \e\left\| n + M^{-1}y\right\|^2}\\
&\leq p_{\re(r)}(n) e^{-\pi \e\left\| n + M^{-1}y\right\|^2},
\end{align}
where $p_{\re(r)}(n)$ is a polynomial independent of $\im(r)$.  Hence, $\Theta_r^{c_1,c_2}(z,\Omega)$ is bounded on the line $\re(r) = \s$ by $\sum_{n\in\Z^g} p_{\s}(n) e^{-\pi \e\left\| n + M^{-1}y\right\|^2}$. It follows that it is bounded on any vertical strip.  Along with periodicity, this implies that $\Theta_r^{c_1,c_2}(z,\Omega)$ as a function of $r$ is bounded and entire, thus constant.
\end{proof}

\section{Definite zeta functions and real analytic Eisenstein series}\label{sec:defzeta}

We will now consider \textit{definite zeta functions}---the Mellin transforms of definite theta functions---in preparation for studying the Mellin transforms of indefinite theta functions in the next section.  In dimension 2, definite zeta functions specialise to real analytic Eisenstein series for the congruence subgroup $\Gamma_1(N)$ (which specialise further to ray class zeta functions of imaginary quadratic ideal classes).

\subsection{Definition and Dirichlet series expansion}

We define the definite zeta function as a Mellin transform of the indefinite theta null with real characteristics.
\begin{defn}\label{defn:defzeta}
Let $\Omega \in \HH^{(0)}_g$ and $p,q \in \R^g$.  The \textit{definite zeta function} is
\begin{equation}
\widehat{\zeta}_{p,q}(\Omega,s) := 
\left\{
\begin{array}{ll}
\int_0^\infty \Theta_{p,q}(t\Omega)t^s\frac{dt}{t} & \mbox{if $q \nin \Z^g$}, \\
\int_0^\infty \left(\Theta_{p,q}(t\Omega)-1\right)t^s\frac{dt}{t} & \mbox{if $q \in \Z^g$}.
\end{array}
\right.
\end{equation}
\end{defn}
By direct calculation, $\widehat{\zeta}_{p,q}(\Omega,s)$ has a Dirichlet series expansion.
\begin{align}
\widehat{\zeta}_{p,q}(\Omega,s) &= (2\pi)^{-s} \Gamma(s)\sum_{\substack{n \in \Z^g \\  n \neq -q}} e(p^\top(n+q))Q_{-i\Omega}(n+q)^{-s},
\end{align}
where $Q_{-i \Omega}(n+q)^{-s}$ is defined using the standard branch of the logarithm (with a branch cut on the negative real axis).

\subsection{Specialisation to real analytic Eisenstein series}

Now, suppose $g=2$, $\Omega = i M$ for some real symmetric, positive definite matrix $M$, $p=\coltwo{0}{0}$, and $q \nin \Z^2$.  Then the definite zeta function may be written as follows.
\begin{align}
\widehat{\zeta}_{0,q}(\Omega,s) 
&= (2\pi)^{-s} \Gamma(s)\sum_{n \in \Z^2} Q_{M}(n+q)^{-s} \\
&= (2\pi)^{-s} \Gamma(s)\sum_{n \in \Z^2+q} Q_{M}(n)^{-s}.
\end{align}
Up to scaling, $M$ is of the form $M = \frac{1}{\im(\t)}\mattwo{1}{\re(\t)}{\re(\t)}{\t\ol{\t}}$ for some $\tau \in \HH$; scaling $M$ by $\l \in \R$ simply scales $\widehat{\zeta}_{p,q}(\Omega,s)$ by $\l^{-s}$, so we assume $M$ is of this form.  Write 
\begin{align}
Q_{M}\coltwo{n_1}{n_2} 
&= \frac{1}{2\im(\tau)}\left(n_1^2+2\re{\t}n_1n_2+\t\ol{\t}n_2^2\right) \\
&= \frac{1}{2\im(\tau)}\abs{n_1+n_2\t}^2
\end{align}
Thus,
\begin{align}
\widehat{\zeta}_{0,q}(\Omega,s) 
&= \pi^{-s} \Gamma(s)\im(\t)^s\sum_{\scriptsize{\coltwo{n_1}{n_2} \in \Z^2+q}} \abs{n_1\tau + n_2}^{-2s}.
\end{align}
If $q \in \Q^2$ and the gcd of the denominators of the entries of $q$ is $N$, this is essentially an Eisenstein series of associated to $\Gamma_1(N)$.  
Choose $k,\ell \in \Z$ such that $q \con \coltwo{k/N}{\ell/N} \Mod{1}$ and $\gcd(k,\ell)=1$.  Then, we have
\begin{align}
\widehat{\zeta}_{0,q}(\Omega,s) 
&= (\pi N)^{-s} \Gamma(s)\im(\t)^s\sum_{\substack{c \con k \Mod{N} \\ d \con \ell \Mod{N}}} \abs{c\tau + d}^{-2s}. \label{eq:arf1}
\end{align}
The Eisenstein series associated to the cusp $\infty$ of $\Gamma_1(N)$ is
\begin{align}
E^\infty_{\Gamma_1(N)}(\t,s)
&= \sum_{\gamma \in \Gamma_1^\infty(N) \backslash \Gamma_1(N)} \im(\gamma \cdot \tau)^s \\
&= \im(\t)^s\sum_{\substack{c \con 0 \Mod{N} \\ d \con 1 \Mod{N} \\ \gcd(c,d)=1}} \abs{c\tau+d}^{-2s} \\
&= \im(\t)^s\frac{\prod_{p|N} (1-p^{-s})}{\zeta(s)}\sum_{\substack{c \con 0 \Mod{N} \\ d \con 1 \Mod{N}}} \abs{c\tau+d}^{-2s}.
\end{align}
Here, $\Gamma_1^\infty(N)$ is the stabiliser of $\infty$ under the fractional linear transformation action; that is, $\Gamma_1^\infty(N) = \left\{\pm\mattwo{1}{n}{0}{1} : n \in \Z\right\}$.

Choose $u,v \in \Z$ such that $\det\mattwo{u}{v}{k}{\ell}=1$.  We have
\begin{align}
\frac{\zeta(s)}{\prod_{p|N} (1-p^{-s})}
E^\infty_{\Gamma_1(N)}\left(\frac{u\tau+v}{k\tau+\ell},s\right)
&= \im\left(\frac{u\tau+v}{k\tau+\ell}\right)^s\sum_{\substack{c \con 0 \Mod{N} \\ d \con 1 \Mod{N}}} \abs{c\left(\frac{u\tau+v}{k\tau+\ell}\right)+d}^{-2s} \\
&= \im\left(\tau\right)^s\sum_{\substack{c \con 0 \Mod{N} \\ d \con 1 \Mod{N}}} \abs{(cu+dk)\tau+(cv+d\ell)}^{-2s} \\
&= \im\left(\tau\right)^s\sum_{\substack{c' \con k \Mod{N} \\ d' \con \ell \Mod{N}}} \abs{c'\tau+d'}^{-2s}. \label{eq:arf2}
\end{align}
Combining \cref{eq:arf1} and \cref{eq:arf2}, we see that
\begin{equation}
\widehat{\zeta}_{0,q}(\Omega,s) 
= \frac{(\pi N)^{-s} \Gamma(s) \zeta(s)}{\prod_{p|N} (1-p^{-s})} E^\infty_{\Gamma_1(N)}\left(\frac{u\tau+v}{k\tau+\ell},s\right).
\end{equation}

\section{Indefinite zeta functions: definition, analytic continuation, and functional equation}\label{sec:indefzeta}

We now turn our attention to the primary objects of interest, \textit{(completed) indefinite zeta functions}---the Mellin transforms of indefinite theta functions. We will generally omit the word ``completed'' when discussing these functions.

As usual, let $\Omega \in \HH^{(1)}_g$, $p,q \in \R^g$, $c_1, c_2 \in \C^g$, $\ol{c_1}^\top M c_1 < 0$, $\ol{c_2}^\top M c_2 < 0$.

We define the indefinite zeta function using a Mellin transform of the indefinite theta function with characteristics.
\begin{defn}\label{defn:indefzeta}
The (completed) indefinite zeta function is
\begin{equation}
\widehat{\zeta}^{c_1,c_2}_{p,q}(\Omega,s) := \int_0^\infty \Theta^{c_1,c_2}_{p,q}(t\Omega)t^s\frac{dt}{t}. \label{eq:indefzdef}
\end{equation}
\end{defn}
The terminology ``zeta function'' here should not be taken to mean that $\widehat{\zeta}^{c_1,c_2}_{p,q}(\Omega,s)$ has a Dirichlet series---it (usually) doesn't (although it does have an analogous series expansion involving hypergeometric functions, as we'll see in \Cref{sec:3series}).  Rather, we think of it as a zeta function by analogy with the definite case, and (as we'll see) because is sometimes specialises to certain classical zeta functions.

By defining the zeta function as a Mellin transform, we've set things up so that a proof of the functional equation \Cref{thm:zetafun} is a natural first step.  Analytic continuation and a functional equation will follow from \Cref{thm:omega} by standard techniques.  Our analytic continuation also gives an expression that converges quickly everywhere and is therefore useful for numerical computation, unlike \cref{eq:indefzdef} or the series expansion in \Cref{sec:3series}.
\begin{thm:zetafun}%\label{thm:zetafun}
The function $\widehat{\zeta}^{c_1,c_2}_{p,q}(\Omega,s)$ may be analytically continued to an entire function on $\C$.
It satisfies the functional equation 
\begin{equation}
\widehat{\zeta}^{c_1,c_2}_{p,q}\left(\Omega,\frac{g}{2}-s\right) =  \frac{e(p^\top q)}{\sqrt{\det(-i\Omega)}} \widehat{\zeta}^{\ol\Omega c_1,\ol \Omega c_2}_{-q,p}\left(-\Omega^{-1},s\right).
\end{equation}
\end{thm:zetafun}
\begin{proof}
Fix $r > 0$, and split up the Mellin transform integral into two pieces,  
\begin{align}
\widehat{\zeta}^{c_1,c_2}_{p,q}\left(\Omega,s\right) &= \int_0^\infty \Theta^{c_1,c_2}_{p,q}(t \Omega)t^s \frac{dt}{t}\\
                                                &= \int_r^\infty\Theta^{c_1,c_2}_{p,q}(t \Omega)t^s \frac{dt}{t} + \int_0^r \Theta^{c_1,c_2}_{p,q}(t \Omega)t^s \frac{dt}{t}.
\end{align}
Replacing $t$ by $t^{-1}$, and then using part (3) of \Cref{thm:omega}, the second integral is
\begin{align}
\int_0^r \Theta^{c_1,c_2}_{p,q}(t \Omega)t^s \frac{dt}{t} &= \int_{r^{-1}}^\infty \Theta^{c_1,c_2}_{p,q}(t^{-1} \Omega)t^{-s} \frac{dt}{t}\\
                                                  &= \int_{r^{-1}}^\infty \frac{e(p^\top q)}{\sqrt{\det(-it \Omega)}} \Theta^{t\ol\Omega c_1,t\ol\Omega c_2}_{-q,p}(-(t^{-1}\Omega)^{-1}) t^{-s} \frac{dt}{t}\\
                                                  &= \frac{e(p^\top q)}{\sqrt{\det(-i \Omega)}} \int_{r^{-1}}^\infty \Theta^{\ol\Omega c_1,\ol\Omega c_2}_{-q,p}(t(-\Omega^{-1}))t^{\frac{g}{2}-s} \frac{dt}{t}.
\end{align}
(Recall that scaling the $c_j$ does not affect the value of $\Theta^{c_1,c_2}_{p,q}(\Omega)$.)  Putting it all together, we have
\begin{align}\label{eq:Thetacont}
\widehat{\zeta}^{c_1,c_2}_{p,q}\left(\Omega,s\right) &= \int_r^\infty\Theta^{c_1,c_2}_{p,q}(t \Omega)t^s \frac{dt}{t} \nn \\ &\ \ \ + \frac{e(p^\top q)}{\sqrt{\det(-i \Omega)}} \int_{r^{-1}}^\infty \Theta^{\ol\Omega c_1,\ol\Omega c_2}_{-q,p}(t(-\Omega^{-1}))t^{\frac{g}{2}-s} \frac{dt}{t}.
\end{align}
As we showed in the proof of \Cref{prop:conv}, the $\Theta$-functions in both integrals decay exponentially as $t \to \infty$, so the right-hand side converges for all $s \in \C$.  The right-hand side is obviously analytic for all $s \in \C$, so we've analytically continued $\widehat{\zeta}^{c_1,c_2}_{p,q}(\Omega,s)$ to an entire function of $s$.
Finally, we must prove the functional equation.  If we plug $\frac{g}{2}-s$ for $s$ in \cref{eq:Thetacont}, factor out the coefficient of the second term, and switch the order of the two terms, we obtain
\begin{align}
\widehat{\zeta}^{c_1,c_2}_{p,q}\left(\Omega,\frac{g}{2}-s\right) =& \  \frac{e(p^\top q)}{\sqrt{\det(-i\Omega)}} \left(\int_{r^{-1}}^\infty \Theta^{\ol\Omega c_1,\ol\Omega c_2}_{-q,p}(t(-\Omega^{-1}))t^{s} \frac{dt}{t} \right. \nn\\ & \ \left. - \frac{e(-p^\top q)}{\sqrt{\det(i\Omega^{-1})}} \int_r^\infty \Theta^{c_1,c_2}_{p,q}(t \Omega)t^{\frac{g}{2}-s} \frac{dt}{t}\right).\label{eq:duck1}
\end{align}
Reusing \cref{eq:Thetacont} on $\widehat{\zeta}^{\ol\Omega c_1,\ol \Omega c_2}_{-q,p}\left(-\Omega^{-1},s\right)$, and
appealing to the fact that $\Theta_{p,q}^{c_1,c_2}(\Omega) = -\Theta_{-p,-q}^{c_1,c_2}(\Omega)$, we have
\begin{align}\label{eq:duck2}
\widehat{\zeta}^{\ol\Omega c_1,\ol \Omega c_2}_{-q,p}\left(-\Omega^{-1},s\right) &= \int_{r^{-1}}^\infty \Theta^{\ol\Omega c_1,\ol\Omega c_2}_{-q,p}(t(-\Omega^{-1}))t^{s} \frac{dt}{t} \nn \\ &\ \ \ - \frac{e(-p^\top q)}{\sqrt{\det(i\Omega^{-1})}} \int_r^\infty \Theta^{c_1,c_2}_{p,q}(t \Omega)t^{\frac{g}{2}-s} \frac{dt}{t}.
\end{align}
The functional equation now follows from \cref{eq:duck1} and \cref{eq:duck2}.
\end{proof}

The formula for the analytic continuation is useful in itself.  In particular, we have used this formula for computer calculations, as it may be used to compute the indefinite zeta function to arbitrary precision in polynomial time.
\begin{cor}
The following expression is valid on the entire $s$-plane.
\begin{align}
\widehat{\zeta}^{c_1,c_2}_{p,q}\left(\Omega,s\right) &= \int_r^\infty\Theta^{c_1,c_2}_{p,q}(t \Omega)t^s \frac{dt}{t} \nn \\ &\ \ \ + \frac{e(p^\top q)}{\sqrt{\det(-i \Omega)}} \int_{r^{-1}}^\infty \Theta^{\ol\Omega c_1,\ol\Omega c_2}_{-q,p}(t(-\Omega^{-1}))t^{\frac{g}{2}-s} \frac{dt}{t}.\label{eq:Thetacont2}
\end{align}
\end{cor}
\begin{proof}
The is \cref{eq:Thetacont}.
\end{proof}

\section{Series expansion of indefinite zeta function}\label{sec:3series}

In this section, we give a series expansion for indefinite zeta functions, under the assumption that $c_1$ and $c_2$ are real.  Specifically, we write $\widehat{\zeta}^{c_1,c_2}_{p,q}(\Omega,s)$ as a sum of three series, the first of which is a Dirichlet series and the others of which involve hypergeometric functions.  This expansion is related to the decomposition of a weak harmonic Maass form into its holomorphic ``mock'' piece and a nonholomorphic piece obtained from a ``shadow'' form in another weight.  %However, we don't describe the relationship here.

To proceed, we will need to introduce some special functions and review some of their properties.

\subsection{Hypergeometric functions and modified beta functions}

Let $a,b,c$ be complex numbers, $c$ not a negative integer or zero.  If $z \in \C$ with $|z|<1$, the power series
\begin{equation}\label{eq:hf}
\hf(a,b;c;z) := \sum_{n=0}^\infty \frac{(a)_n(b)_n}{(c)_n} \cdot \frac{z^n}{n!}
\end{equation}
converges.  Here we are using the Pochhammer symbol $(w)_n := w(w+1)\cdots (w+n-1)$.
\begin{prop}\label{hypergeocont}
There is an identity
\begin{equation}
\hf(a,b;c;z) = (1-z)^{-b}\hf\left(b,c-a; c; \frac{z}{z-1}\right),
\end{equation}
valid about $z=0$ and using the principal branch for $(1-z)^{-b}$.
\end{prop}
\begin{proof}
This is part of Theorem 2.2.5 of \cite{special}.
\end{proof} 
Using this identity, we extend the domain of definition of $\hf(a,b;c;x)$ from the unit disc $\{|z|<1\}$ to the union of the unit disc and a half-plane $\{|z|<1\} \cup \{\re(z)<\foh\}$. We interpret $(1-z)^{-b} = \exp(-b\log(1-z))$ with the logarithm having a branch cut along the negative real axis.  At the boundary point $z=1$, the hypergeometric series converges when $\re(c) > \re(a+b)$, and its evaluation is a classical theorem of Gauss.
\begin{prop}\label{gausshypergeo}
If $\re(c) > \re(a+b)$, then
\begin{equation}
\hf(a,b;c;1) = \frac{\Gamma(c)\Gamma(c-a-b)}{\Gamma(c-a)\Gamma(c-b)}.
\end{equation}
\end{prop}
\begin{proof}
This is Theorem 2.2.2 of \cite{special}.
\end{proof}

Of particular interest to us will be a special hypergeometric function which is a modified version of the beta function.
\begin{defn}
Let $x > 0$ and $a,b \in \C$.  The \textit{beta function} is
\begin{equation}
B(x;a,b) := \int_0^x t^{a-1}(1-t)^{b-1}\,dt,
\end{equation}
and the \textit{modified beta function} is
\begin{equation}
\tB(x;a,b) := \int_0^x t^{a-1}(1+t)^{b-1}\,dt.
\end{equation}
\end{defn}
The following proposition enumerates some properties of the modified beta function.
\begin{prop}\label{hfrelation}
Let $x > 0$, and let $a, b$ be complex numbers with $\re(a),\re(b)>0$ and $\re(a+b)<1$.  Then,  
\begin{enumerate}
\item[(1)] 
$\ds \tB(x;a,b) = B\left(\frac{x}{x+1};a,1-a-b\right)$,
\item[(2)] 
$\ds \tB(x;a,b) = \frac{1}{a}x^a\hf(a,1-b;a+1;-x)$,
\item[(3)] 
$\ds \tB\left(\frac{1}{x};a,b\right) = \frac{\Gamma(a)\Gamma(1-a-b)}{\Gamma(1-b)} - \tB(x;1-a-b,b)$, and
\item[(4)] 
$\ds \tB(+\infty;a,b) = B(1;a,1-a-b) = \frac{\Gamma(a)\Gamma(1-a-b)}{\Gamma(1-b)}$.
\end{enumerate}
\end{prop}
\begin{proof}
To prove (1), we use the substitution $t=\frac{u}{1-u}$.
\begin{align}
\tB(x;a,b) &= \int_0^x t^{a-1}(1+t)^{b-1}\,dt\\
             &= \int_0^{\frac{x}{x+1}} \left(\frac{u}{1-u}\right)^{a-1}\left(1+\frac{u}{1-u}\right)^{b-1}\,\frac{du}{(1-u)^2}\\
             &= \int_0^{\frac{x}{x+1}} u^{a-1}(1-u)^{-a-b}\,du\\
             &= B\left(\frac{x}{x+1};a,1-a-b\right).
\end{align}
To prove (2), expand $G(x;a,b)$ as a power series in $x$ (up to a non-integral power).
\begin{align}
\tB(x;a,b) &= \int_0^x t^{a-1}(1+t)^{b-1}\,dt\\
             &= \int_0^x \sum_{n=0}^\infty {b-1 \choose n} t^{n+a-1}\,dt\\
             &= \sum_{n=0}^\infty {b-1 \choose n} \frac{1}{n+a}x^{n+a}\\
             &= \sum_{n=0}^\infty \frac{(b-n)\cdot(b-n+1)\cdots(b-1)}{n!} \cdot \frac{1}{n+a}x^{n+a}\\
             &= x^a\sum_{n=0}^\infty \frac{(-1)^n(1-b)\cdot(2-b)\cdots(n-b)}{n+a} \cdot \frac{x^{n}}{n!}\\
             &= x^a\sum_{n=0}^\infty \frac{(a)_n(1-b)_n}{a(a+1)_n} \cdot \frac{(-x)^{n}}{n!}\\
             &= \frac{1}{a}x^a\hf(a,1-b;a+1;-x).
\end{align}
To prove (3), use the substitution $t=\frac{1}{u}$.
\begin{align}
\tB\left(\frac{1}{x};a,b\right) &= \int_0^{1/x} t^{a-1}(1+t)^{b-1}\,dt\\
                                          &= \int_\infty^x u^{-a+1}\left(1+\frac{1}{u}\right)^{b-1}\,\left(-\frac{du}{u^2}\right)\\
                                          &= \int_x^\infty u^{-a-b}(1+u)^{b-1}\,du\\
                                          &= G(+\infty,1-a-b,b) - G(x,1-a-b,b)
\end{align}
To complete the proof of (3), we need to prove (4).  Note that it follows from (4) that $\tB(+\infty,1-a-b,b) = \frac{\Gamma(a)\Gamma(1-a-b)}{\Gamma(1-b)}$.  The first equality of (4) follows from (1) with $x \to +\infty$; we will now derive the second.
By (2),
\begin{align}
\tB(x;a,b) &= \frac{1}{a}x^a\hf(a,1-b;a+1;-x) \\
             &= \frac{1}{a}x^a\hf(1-b,a;a+1;-x) \\
             &= \frac{1}{a}x^a \cdot (1-(-x))^{-a}\hf\left(a,(a+1)-(1-b); a+1; \frac{-x}{(-x)-1}\right) \label{eq:ntl}\\
             &= \frac{1}{a}\left(\frac{x}{1+x}\right)^a\hf\left(a,a+b; a+1; \frac{x}{x+1}\right)
\end{align}
\Cref{hypergeocont} was used in \cref{eq:ntl}.  
Sending $x \to +\infty$ and applying \Cref{gausshypergeo} yields the second equality of (4).
\end{proof}

\begin{lem}\label{mellinlem}
Let $\lambda, \mu > 0$, and $\re(s) > 0$.  Then
\begin{align}
%\int_0^\infty \beta(\lambda t)\exp(-\mu t)t^s\frac{dt}{t} 
%&= \pi^{-1/2} \mu^{-s} \Gamma\left(s+\foh\right)\widetilde{B}\left(\frac{\mu}{\pi \lambda};s,\foh-s\right), \label{lem1part1} \\ 
\int_0^\infty \eE(\sqrt{\lambda t})\exp(-\mu t)t^s\frac{dt}{t} 
&= \foh \pi^{-1/2} \mu^{-s} \Gamma\left(s+\foh\right)\widetilde{B}\left(\frac{\pi \lambda}{\mu};\foh,\foh-s\right). \label{lem1part2}
\end{align}
\end{lem}
\begin{proof}
First of all, note that the left-hand side of \Cref{lem1part2} converges: The integrand is $\exp(-O(t))$ as $t \to \infty$ and $O\left(t^{\re{s}-\foh}\right)$ as $t \to 0$.
Write $\eE(\sqrt{\lambda t}) = \foh\int_0^{\lambda t} u^{-1/2}e^{-\pi u}du$.  The left-hand side of \Cref{lem1part2} may be rewritten, using the substitution $u=\frac{\mu t v}{\pi}$ in the inner integral, as
\begin{align}
\int_0^\infty \eE(\sqrt{\lambda t})\exp(-\mu t)t^s\frac{dt}{t} 
&= \foh\int_0^\infty \int_0^{\lambda t} u^{-1/2}e^{-(\pi u + \mu t)}t^s\,du\frac{dt}{t}\\
&= \foh\int_0^\infty \int_0^{\frac{\pi \lambda}{\mu}} \left(\frac{\mu t v}{\pi}\right)^{-1/2}e^{-(\mu t v + \mu t)}t^s\,\frac{\mu t}{\pi}dv\frac{dt}{t}.
\end{align}
The double integral is absolutely convergent (indeed, the integrand is nonnegative, and we already showed convergence), so we may swap the integrals.  We compute
\begin{align}
\int_0^\infty \eE(\sqrt{\lambda t})\exp(-\mu t)t^s\frac{dt}{t}
&= \foh\left(\frac{\mu}{\pi}\right)^{1/2}\int_0^{\frac{\pi \lambda}{\mu}} v^{-1/2} \left(\int_0^\infty e^{-\mu t(v+1)}t^{s+\foh}\,\frac{dt}{t}\right)\,dv\\
&= \foh\left(\frac{\mu}{\pi}\right)^{1/2}\int_0^{\frac{\pi \lambda}{\mu}} v^{-1/2} \left(\Gamma\left(s+\foh\right)\left(\mu(v+1)\right)^{-(s+\foh)}\right)\,dv\\
&= \foh\pi^{-1/2}\mu^{-s}\Gamma\left(s+\foh\right)\int_0^{\frac{\pi \lambda}{\mu}} v^{-1/2}(v+1)^{-(s+\foh)}\,dv\\
&= \foh\pi^{-1/2}\mu^{-s}\Gamma\left(s+\foh\right)\tB\left(\frac{\pi \lambda}{\mu};\foh,\foh-s\right).
\end{align}
This proves \Cref{lem1part2}.  
\end{proof}

\begin{lem}\label{lem:laneboy}
Let $\nu_1, \nu_2 \in \R$ and $\mu \in \C$ satisfying $\re(\mu)>-\pi\max\{\nu_1^2,\nu_2^2\}$ if $\sgn(\nu_1) = \sgn(\nu_2)$ and $\re(\mu)>0$ otherwise.  Then,
\begin{align}
&\int_0^\infty \left.\eE\left(\nu t^{1/2}\right)\right|_{\nu=\nu_1}^{\nu_2} \exp(-\mu t)t^s\frac{dt}{t} \nn\\
&= \foh\left(\sgn(\nu_2)-\sgn(\nu_1)\right)\Gamma(s)\mu^{-s} \nn\\ 
&\ \ \ - \frac{\sgn(\nu_2)}{2s}\pi^{-\left(s+\foh\right)}\Gamma\left(s+\foh\right)\abs{\nu_2}^{-2s}\hf\left(s,s+\foh,s+1;-\frac{\mu}{\pi\nu_2^2}\right) \nn\\
&\ \ \ + \frac{\sgn(\nu_1)}{2s}\pi^{-\left(s+\foh\right)}\Gamma\left(s+\foh\right)\abs{\nu_1}^{-2s}\hf\left(s,s+\foh,s+1;-\frac{\mu}{\pi\nu_1^2}\right). \label{eq:laneboy}
\end{align}
\end{lem}
\begin{proof}
Initially, consider $\lambda, \mu >0$, as in \Cref{mellinlem}.  We have
\begin{align}
&\int_0^\infty \eE\left(\sqrt{\l t}\right)\exp(-\mu t)t^s\frac{dt}{t} \nn \\ &= \frac{1}{2}\pi^{-\foh}\mu^{-s}\Gamma\left(s+\foh\right)\tB\left(\frac{\pi\l}{\mu};\foh,\foh-s\right) \\
&= \foh\pi^{-\foh}\mu^{-s}\Gamma\left(s+\foh\right)\left(\frac{\Gamma(\foh)\Gamma(s)}{\Gamma(s+\foh)} - \tB\left(\frac{\mu}{\pi\l};s,1-s\right)\right) \\
&= \foh\Gamma(s)\mu^{-s} - \foh\pi^{-\foh}\mu^{-s}\Gamma\left(s+\foh\right)\tB\left(\frac{\mu}{\pi\l};s,1-s\right) \\
&= \foh \Gamma(s)\mu^{-s}  - \frac{1}{2s}\pi^{-(s+\foh)}\Gamma\left(s+\foh\right)\l^{-s}\hf\left(s,s+\foh,s+1;-\frac{\mu}{\pi\l}\right),
\end{align}
using parts (2) and (3) of \Cref{hfrelation}.  \Cref{eq:laneboy} follows for positive real $\mu$.  But the integral on the left-hand side of \cref{eq:laneboy} converges for $\re(\mu)>-\pi\max\{\nu_1^2,\nu_2^2\}$ if $\sgn(\nu_1) = \sgn(\nu_2)$ and $\re(\mu)>0$ otherwise, and both sides are analytic functions in $\mu$ on this domain.  Thus, \cref{eq:laneboy} holds in general by analytic continuation.
\end{proof}

\subsection{The series expansion}

We are now ready to prove \Cref{thm:splitsy}, which we first restate here for convenience.
\begin{thm:splitsy}%\label{thm:splitsy}
If $c_1,c_2 \in \R^g$, and $\re(s)>1$, then the indefinite zeta function may be written as
\begin{equation}
\widehat{\zeta}^{c_1,c_2}_{p,q}(\Omega,s) = \pi^{-s}\Gamma(s)\zeta^{c_1,c_2}_{p,q}(\Omega,s)-\pi^{-(s+\foh)}\Gamma\left(s+\foh\right)\left(\xi^{c_2}_{p,q}(\Omega,s)-\xi^{c_1}_{p,q}(\Omega,s)\right),
\end{equation}
where $M = \im(\Omega)$,
\begin{align}
\zeta^{c_1,c_2}_{p,q}(\Omega,s) &= \foh\sum_{n \in \Z^g+q}\left(\sgn(c_1^\top Mn) - \sgn(c_2^\top Mn)\right)e\left(p^\top n\right)Q_{-i\Omega}(n)^{-s},
\end{align}
and
\begin{align}
\xi^{c}_{p,q}(\Omega,s) &= \foh\sum_{\nu \in \Z^g+q} \left(\sgn(c^\top Mn)e\left(p^\top n\right)\left(\frac{(c^\top Mn)^2}{Q_M(c)}\right)^{-s}\right.\nn \\&\ \ \ \ \ \ \ \ \ \ \ \ \ \ \ \ \ \ \times\left.\hf\left(s,s+\foh,s+1;\frac{2Q_M(c)Q_{-i\Omega}(n)}{(c^\top Mn)^2}\right)\right).
\end{align}
\end{thm:splitsy}
\begin{proof}
Take the Mellin transform of the theta series term-by-term, and apply \Cref{lem:laneboy}.  
Note that the series for $\xi^{c}_{p,q}(\Omega,s)$ converges absolutely, so the series may be split up like this.
\end{proof}

The function $\zeta^{c_1,c_2}_{p,q}(\Omega,s)$ here is a Dirichlet series summed over a double cone, with any lattice points on the boundary of the cone weighted by $\foh$.  The coefficients of the terms are $\pm e\left(p^\top n\right)$, where the sign is determined by whether one is in the positive or negative part of the double cone.

\begin{thm}\label{thm:zetastable}
Suppose 
$(c_1,c_2,p+\Omega q,\Omega)$
is $P$-stable.  Then, $\xi^{c_1}_{p,q}(\Omega,s)=\xi^{c_2}_{p,q}(\Omega,s)$ and $\widehat{\zeta}^{c_1,c_2}_{p,q}(\Omega,s) = \pi^{-s}\Gamma(s)\zeta^{c_1,c_2}_{p,q}(\Omega,s)$.
\end{thm}
\begin{proof}
The equality of the $\xi^{c_j}_{p,q}(\Omega,s)$ follows by the substitution $n \mapsto Pn$ and the definition of $P$-stability.  The equation 
\begin{equation}
\widehat{\zeta}^{c_1,c_2}_{p,q}(\Omega,s) = \pi^{-s}\Gamma(s)\zeta^{c_1,c_2}_{p,q}(\Omega,s)
\end{equation}
then follows from \Cref{thm:splitsy}.
\end{proof}

\section{Zeta functions of ray ideal classes in real quadratic fields}\label{sec:quadratic}

In this section, we will specialise indefinite zeta functions to obtain certain zeta functions to obtain certain zeta functions attached to real quadratic fields.
We define two Dirichlet series, $\zeta(s,A)$ and $Z_A(s)$, attached to a ray ideal class $A$ of the ring of integers of a number field.
\begin{defn}[Ray class zeta function]\label{defn:rayzeta}
Let $K$ be any number field and $\cc$ an ideal of the maximal order $\OO_K$.  Let $S$ be a subset of the real places of $K$ (i.e., the embeddings $K \inj \R$).  Let $A$ be a ray ideal class modulo $\cc S$, that is, an element of the group
\begin{equation}\label{eq:Cl}
\Cl_{\cc S} := \Cl_{\cc S}(\OO_K) := \frac{\{\mbox{nonzero fractional ideals of $\OO_K$ coprime to $\cc$}\}}{\{a\OO_K : a \con 1 \Mod{\cc} \mbox{ and $a$ is positive at each place in $S$}\}}.
\end{equation}
Define the \textit{zeta function of $A$} to be
\begin{equation}
\zeta(s,A) := \sum_{\aa \in A} N(\aa)^{-s}.
\end{equation}
\end{defn}
This function has a simple pole at $s=1$ with residue independent of $A$.  The pole may be eliminated by considering the function $Z_A(s)$, defined as follows.
\begin{defn}[Differenced ray class zeta function]\label{defn:diffzeta}
Let $R$ be the element of $\Cl_{\cc S}$ defined by 
\begin{equation}
R:= \{a\OO_K : a \con -1 \Mod{\cc} \mbox{ and $a$ is positive at each place in $S$}\}.
\end{equation}
Define the \textit{differenced zeta function of $A$} to be
\begin{equation}
Z_A(s) := \zeta(s,A) - \zeta(s,RA).
\end{equation}
\end{defn}
The function $Z_A(s)$ is holomorphic at $s=1$.

Now, specialise to the case where $K = \Q(\sqrt{D})$ be a real quadratic field of discriminant $D$.  Let $\OO_K$ be the maximal order of $K$, and let $\cc$ be an ideal of $\OO_K$.  Let $A$ be a narrow ray ideal class modulo $\cc$, that is, an element of the group $\Cl_{\cc\infty_1\infty_2}(\OO_K)$.
We show, as promised in the introduction, that the indefinite zeta function specialises to the $L$-series $Z_A(s)$ attached to a ray class of an order in a real quadratic field.

\begin{thm:special}%\label{thm:special}
For each $A \in \Cl_{\cc\infty_1\infty_2}$ and integral ideal $\bb \in A^{-1}$, there exists a real symmetric $2 \times 2$ matrix $M$, vectors $c_1, c_2 \in \R^2$,and $q \in \Q^2$ such that
\begin{equation}
(2\pi N(\bb))^{-s}\Gamma(s)Z_A(s) = \widehat{\zeta}^{c_1,c_2}_{0,q}(iM,s).
\end{equation}
\end{thm:special}
\begin{proof}
The differenced zeta function $Z_A(s)$ is
\begin{equation}
Z_A(s) = \sum_{\aa\in A} N(\aa)^{-s} - \sum_{\aa\in RA} N(\aa)^{-s}.
\end{equation}
We have
\begin{align}
N(\bb)^{-s}Z_A(s) 
&= \sum_{\aa\in A} N(\bb\aa)^{-s} - \sum_{\aa\in RA} N(\bb\aa)^{-s} \\
&= \sum_{\substack{b \in \bb \\ (b) \in I \\ \scriptsize\mbox{up to units}}} N(b)^{-s} - \sum_{\substack{b \in \bb \\ (b) \in R \\ \scriptsize\mbox{up to units}}} N(b)^{-s}.
\end{align}
Write $\bb\cc = \g_1\Z+\g_2\Z$.  The norm form $N(n_1\g_1+n_2\g_2) = Q_M\coltwo{n_1}{n_2}$ for some real symmetric matrix $M$ with integer coefficients.  The signature of $M$ is $(1,1)$, just like the norm form for $K$.  Since $\bb$ and $\cc$ are relatively prime (meaning $\bb+\cc = \oO_K$), there exists by the Chinese remainder theorem some $b_0 \in \oO_K$ such that $b \con b_0 \Mod{\bb\cc}$ if and only if $b \con 0 \Mod{\bb}$ and $b \con 1 \Mod{\cc}$.  Express $b_0 = p_1\g_1+p_2\g_2$ for rational numbers $p_1,p_2$, and set $p=\coltwo{p_1}{p_2}$.

Let $\e_0$ be the fundamental unit of $\oO_K$, and let $\e$ ($=\e_0^k$ for some $k$) be the smallest totally positive unit of $\oO_K$ greater than 1 such that $\e \con 1 \Mod{\cc}$.

Choose any $c_1 \in \R^2$ such that $Q_M(c_1)<0$.  Let $P$ be the matrix describing the linear action of $\e$ on $\bb$ by multiplication, i.e., $\e(\b^\top n) = \b^\top(Pn)$.  Set $c_2=Pc_1$.

Thus, we have
\begin{align}
N(\b)^{-s}Z_A(s) = \foh \sum_{n \in \Z^2+q} \left(\sgn(c_2^\top M n)-\sgn(c_1^\top M n)\right)Q_M(n). \label{eq:train}
\end{align}
Moreover, 
$(c_1,c_2,p,\Omega)$
is $P$-stable.  So, by \Cref{thm:zetastable}, \cref{eq:train} may be rewritten as
\begin{align}
(2\pi N(\bb))^{-s}\Gamma(s)Z_A(s) = \widehat{\zeta}^{c_1,c_2}_{0,q}(iM,s),
\end{align}
completing the proof.
\end{proof}

\subsection{Example}\label{sec:tree}
Let $K = \Q(\sqrt{3})$, so $\O_K = \Z[\sqrt{3}]$, and let $\cc = 5\O_K$.
The ray class group $\Cl_{\cc\infty_2} \isom \Z/8\Z$.
The fundamental unit $\e = 2+\sqrt{3}$ is totally positive: $\e\e' = 1$.  It has order $3$ modulo $5$: $\e^3 = 26+15\sqrt{3} \con 1 \Mod{5}$.
In this section, we use the analytic continuation \cref{eq:Thetacont2} for indefinite zeta functions to compute $Z_I'(0)$, where $I$ is the principal ray class of $\Cl_{\cc\infty_2}$.

By definition, $Z_I = \zeta(s,I) - \zeta(s,R)$ where
\begin{align}
R &= \{a\oO_K : a \con -1 \Mod{\cc} \mbox{ and $a$ is positive at $\infty_2$}\} \\
   &= \{a\oO_K : a \con 1 \Mod{\cc} \mbox{ and $a$ is negative at $\infty_2$}\}.
\end{align}
Write $I = I_+ \sqcup I_-$ and $R = R_+ \sqcup R_-$, where $I_{\pm}$ and $R_{\pm}$ are the following ray ideal classes in $\Cl_{\cc\infty_1\infty_2}$:
\begin{align}
I_{\pm} &:= \{a\oO_K : a \con 1 \Mod{\cc} \mbox{ and $a$ has sign $\pm$ at $\infty_1$ and $+$ at $\infty_2$}\}, \\
R_{\pm} &:= \{a\oO_K : a \con 1 \Mod{\cc} \mbox{ and $a$ has sign $\pm$ at $\infty_1$ and $-$ at $\infty_2$}\}.
\end{align}

Thus, $Z_I(s) = \zeta(s,I_+) + \zeta(s,I_-) - \zeta(s,R_+) - \zeta(s,R_-)$.  The Galois automorphism $(a_1+a_2\sqrt{3})^\s = (a_1-a_2\sqrt{3})$ defines a norm-preserving bijection between $I_-$ and $R_+$, so the middle terms cancel and
\begin{equation}\label{eq:noplus}
Z_I(s) = \zeta(s,I_+) - \zeta(s,R_-) = Z_{I_+}(s).
\end{equation}

To the principal ray class $I_+$ of $\Cl_{\cc\infty_1\infty_2}$, we associate $\Omega = i M$ where $M = \mattwo{2}{0}{0}{-6}$ and $q = \coltwo{1/5}{0}$.  We may choose $c_1\in \R^2$ arbitrarily so long as $c_1^\top M c_1 < 0$; take $c_1 = \coltwo{0}{1}$.  The left action of $\e$ on $\Z+\sqrt{3}\Z$ is given by the matrix $P=\mattwo{2}{3}{1}{2}$.  By \Cref{thm:special},
\begin{equation}\label{eq:wing}
(2\pi)^{-s}\Gamma(s)Z_{I_+}(s) = \widehat{\zeta}^{c_1,P^3c_1}_{0,q}(iM,s).
\end{equation}
Taking a limit as $s \to 0$, and using \cref{eq:noplus}, \cref{eq:wing} becomes
\begin{equation}
Z_I'(0) = Z_{I_+}'(s) = \widehat{\zeta}^{c_1,P^3c_1}_{0,q}(iM,0).
\end{equation}
For the purpose of making the numerical computation more efficient, we split up the right-hand side as
\begin{align}
Z_I'(0) 
&= \widehat{\zeta}^{c_1,Pc_1}_{0,q}(iM,0) + \widehat{\zeta}^{Pc_1,P^2c_1}_{0,q}(iM,0) + \widehat{\zeta}^{P^2c_1,P^3c_1}_{0,q}(iM,0) \\
&= \widehat{\zeta}^{c_1,Pc_1}_{0,q_0}(iM,0) + \widehat{\zeta}^{c_1,Pc_1}_{0,q_1}(iM,0) + \widehat{\zeta}^{c_1,Pc_1}_{0,q_2}(iM,0),
\end{align}
where $q_0 = q = \frac{1}{5}\coltwo{1}{0}$, $q_1 = q = \frac{1}{5}\coltwo{2}{1}$, and $q_2 = q = \frac{1}{5}\coltwo{2}{4}$ are obtained from the residues of $\e^0, \e^1, \e^2$ modulo $5$.  

Using \cref{eq:Thetacont2}, we computed $Z_I'(0)$ to 100 decimal digits.  The decimal begins
\begin{align}
Z_I'(0) = 1.35863065339220816259511308230\ldots.
\end{align}
The conjectural Stark unit is $\exp(Z_I'(0)) = 3.89086171394307925533764395962\ldots$.  
We used the \texttt{RootApproximant[]} function in Mathematica, which uses lattice basis reduction internally, to find a degree 16 integer polynomial having this number as a root, and we factored that polynomial over $\Q(\sqrt{3})$.  To 100 digits, $\exp(Z_I'(0))$ is equivalent to be the root of the polynomial
\begin{align}
x^8 &- (8 + 5\sqrt{3})x^7 + (53 + 30\sqrt{3})x^6 - (156 + 90\sqrt{3})x^5 + (225 + 130\sqrt{3})x^4\nn\\ 
   &- (156 + 90\sqrt{3})x^3 + (53 + 30\sqrt{3})x^2 - (8 + 5\sqrt{3})x + 1.
\end{align}
We have verified that this root generates the expected class field $H_2$.

We have also computed $Z_I'(0)$ a different way in PARI/GP, using its internal algorithms for computing Hecke L-values.  We obtained the same numerical answer this way.

\section{Acknowledgements}

This research was partially supported by National Science Foundation (USA) grants DMS-1401224, DMS-1701576, and DMS-1045119, and by the Heilbronn Institute for Mathematical Research (UK).

This paper 
incorporates material from the author's PhD thesis \cite{koppthesis}.
Thank you to Jeffrey C. Lagarias for advising my PhD and for many helpful conversations about the content of this paper. Thank you to Marcus Appleby, Jeffrey C. Lagarias, Kartik Prasanna, and two anonymous referees for helpful comments and corrections.

\section{Conflict of interest statement}

On behalf of all authors, the corresponding author states that there is no conflict of interest.

\bibliographystyle{plain}
\bibliography{references}

\end{document}